\documentclass{amsart}

\usepackage{amssymb}

\newtheorem{theorem}{Theorem}[section]
\newtheorem{corollary}[theorem]{Corollary}
\newtheorem{definition}[theorem]{Definition}
\newtheorem{model}[theorem]{Model}
\newtheorem*{cuh}{The Computable Universe Hypothesis}

\begin{document}

\title{Semantics of Computable Physical Models}

\author{Matthew~P. Szudzik}

\begin{abstract}
This article reformulates the theory of computable physical models, previously introduced by the author, as a branch of applied model theory in first-order logic.  It provides a semantic approach to the philosophy of science that incorporates aspects of operationalism and Popper's degrees of falsifiability.
\end{abstract}

\date{28 April 2023}

\maketitle

\section{Introduction}

Consider a sentence that expresses some mathematical statement.  To specify the intended meaning of that sentence, the established practice in mathematical logic is to provide a structure $\mathfrak{A}$ for the language of the sentence.  In the formalism of Enderton~\cite[Sect.\ 2.2]{Enderton2001}, if the sentence is written in a first-order language, then
\begin{enumerate}
\renewcommand{\theenumi}{\roman{enumi}}
\item to each nonlogical predicate symbol $P$ in the language, the structure assigns a set-theoretic relation $P^\mathfrak{A}$;
\item to each function symbol $f$ in the language, the structure assigns a set-theoretic function $f^\mathfrak{A}$;
\item to the quantifier symbol in the language, the structure assigns a nonempty set $\lvert\mathfrak{A}\rvert$; and
\item to each constant symbol in the language, the structure assigns a member of $\lvert\mathfrak{A}\rvert$.
\end{enumerate}
The sentence is understood to be a statement about these set-theoretic objects (that is, these relations, functions, and members of the set $\lvert\mathfrak{A}\rvert$).  In this sense, the mathematical statement is a statement about sets.

Now consider a sentence that expresses a statement about the physical universe.  If this sentence is written in a first-order language, then a structure may assign set-theoretic objects to the symbols in the language, as was done in the previous paragraph.  But this sentence is a statement about the physical universe, not a statement about sets.  To specify the intended meaning of the sentence, we can choose to associate
%
%
a physical meaning with some of those set-theoretic objects.\footnote{
This sort of approach to specifying the meaning of a physical statement is known as the \emph{semantic approach}, and it is often contrasted with the \emph{syntactic approach} of the logical positivists.  Liu~\cite[pp.\ 149--154]{Liu1997} provides brief summaries of both approaches.
}
%
%
%
This association of physical meaning with some of the objects in the structure provides a \emph{physical semantics} for the structure.  We will say that this structure, together with its physical semantics, form a \emph{physical model}.\thinspace\footnote{
%
%
We sometimes omit the adjective ``physical'' when it is clear that the models being discussed are physical models, or when the semantics being discussed are physical semantics.
}

This article is concerned with the physical semantics of computable physical models.  Computable physical models were previously introduced by the author~\cite{Szudzik2012,Szudzik2013} to formalize the notion of a computer model for a physical phenomenon~\cite[pp.\ 482--483]{Szudzik2013}.  We begin the present article by reformulating the definition of a computable physical model so that techniques from mathematical logic (especially from model theory) can be applied more directly.  In particular, in Section~\ref{S:nip-models} we introduce a family of many-sorted first-order languages, and we define the notion of a nonnegative integer physical model.  Some basic properties of these models are discussed in Section~\ref{S:properties}.  Informal models that are commonly encountered in the sciences, including discrete models, continuous models, and statistical models,
%
%
can often be formalized as nonnegative integer physical models.  And the definition of a nonnegative integer physical model allows one to discuss, in a straightforward manner, the computability of these models.  Issues related to computability are discussed in Section~\ref{S:cp-models}.  In particular, we introduce computable physical models, and we provide a formulation of the computable universe hypothesis.  Derived observable quantities are then formalized in Section~\ref{S:definitional}, and in Section~\ref{S:restrictions} we discuss the restriction of a nonnegative integer physical model to a set of possible measurement results.  These concepts are illustrated in an extended example in Section~\ref{S:example}.  Up to this point, we will have only discussed discrete models.  The last three sections of the article are concerned with continuous models and statistical models.  In particular, nonnegative integer physical models that are specified by a nonempty closed set of real numbers are described in Section~\ref{S:real}.  This idea is generalized to topological spaces with countable bases in Section~\ref{S:t-spaces}.  Then, probabilities are introduced to nonnegative integer physical models in Section~\ref{S:probabilities}.

Throughout this article, we assume that readers are familiar with the notational and terminological conventions in Enderton's logic textbook~\cite{Enderton2001}.

\section{Nonnegative Integer Physical Models}\label{S:nip-models}

A language for first-order number theory\footnote{
There are various standard ways to formulate a first-order language of number theory.  We assume that the language has a predicate symbol for equality, function symbols for addition and multiplication, and possibly other nonlogical symbols.  See Enderton~\cite[Sect.\ 3.0]{Enderton2001} and Rogers~\cite[p.\ 96]{Rogers1987}, for example.
} can be regarded as a many-sorted language that has only one sort: the sort $N$ of nonnegative integers.  We define a \emph{nonnegative integer physical language} to be any many-sorted first-order language that can be obtained from a language for first-order number theory by introducing
\begin{enumerate}
\renewcommand{\theenumi}{\roman{enumi}}
\item one or more new sorts, one of which is designated as sort $S$;
\item\label{P:oq-symbols} one or more function symbols of sort $\langle S,N\rangle$ (these are said to be the \emph{symbols for observable quantities}); and
\item\label{P:additional-symbols} zero or more additional symbols of any sort.
\end{enumerate}
The symbols introduced in parts~(\ref{P:oq-symbols}) and~(\ref{P:additional-symbols}) are said to be the \emph{physical symbols} in the language.  Then, a \emph{nonnegative integer physical model} is
\begin{enumerate}
\renewcommand{\theenumi}{\roman{enumi}}
\item a structure $\mathfrak{A}$ for a nonnegative integer physical language, such that the nonlogical symbols of first-order number theory are assigned their traditional set-theoretic meanings (in particular, $\lvert\mathfrak{A}\rvert_N$ is the set $\mathbb{N}$ of nonnegative integers); together with
\item a rule, such that for each symbol for an observable quantity that is in the language of $\mathfrak{A}$, the rule assigns a physical measuring operation
%
%
to the symbol, and this operation encodes the result of each measurement as a nonnegative integer.\footnote{
Although it might be more traditional (for example, see Rosen~\cite[p.\ 26]{Rosen1978}) to use real numbers as the results of measurements, nonnegative integers can be used without any loss of generality because, in actual practice, the result of every measurement is recorded as a finite sequence of symbols chosen from a finite alphabet, and such sequences can be encoded as nonnegative integers.
}
\end{enumerate}
The members of $\lvert\mathfrak{A}\rvert_S$ are said to be the \emph{states} of the nonnegative integer physical model.

If $f$ is a symbol for an observable quantity, we say that $f^\mathfrak{A}$ is the corresponding \emph{observable quantity}, and we use $\mathrm{op}(f)$ to denote the physical measuring operation that is assigned to $f$.  In a \emph{faithful}\thinspace\footnote{
%
%
Our notion of faithfulness plays a role that is similar to that of van Fraassen's \emph{empirical adequacy}~\cite[p.\ 12]{VanFraassen1980}.
} nonnegative integer physical model, $f^\mathfrak{A}$ and $\mathrm{op}(f)$ are associated in the following manner.

\begin{definition}\label{D:faithful}
Let $\mathfrak{A}$ be a structure for a nonnegative integer physical model.  The model is said to be \emph{faithful} if and only if, for each symbol $f$ for an observable quantity and each nonnegative integer $n$, if a measurement result of $\mathrm{op}(f)$ is ever equal to $n$, then there exists a state $s\in\lvert\mathfrak{A}\rvert_S$ such that $f^\mathfrak{A}(s)=n$.
\end{definition}

One example of a nonnegative integer physical model, chosen from particle physics, is the following.
\begin{model}\label{M:baryon}
Consider a two-sorted nonnegative integer physical language with a symbol $f$ for an observable quantity, and with no additional physical symbols.  Let $\mathfrak{A}$ be a structure for this language, where $\lvert\mathfrak{A}\rvert_S$ is the set of all $s\in\mathbb{N}$.  Define $\mathrm{op}(f)$ to be an operation that counts the total number of baryons and antibaryons produced in a collision of two protons, and let $f^\mathfrak{A}(s)=2s+2$.
\end{model}
%
%
\noindent
The law of baryon number conservation implies that this model is faithful.\footnote{
Protons are baryons, and the law of baryon number conservation states that the number of baryons minus the number of antibaryons is always conserved~\cite[Sect.\ 1.6]{Griffiths2010}.  Although violations of this law have been suspected to exist, no violation has ever been observed~\cite[Sect.\ 3.8]{Perkins2000}.
}  And assuming that it is faithful, the model predicts that the total number of baryons and antibaryons produced in a collision of two protons can never be an odd number.

Following the semantic approach of Dalla Chiara Scabia and Toraldo di Francia~\cite[p.\ 5]{DallaChiaraScabia1973}, we insist that ``operations that define a quantity via a measurement procedure need not exclude, indeed necessarily include, a certain amount of data processing.''  Going further, we identify the concept of a measuring operation with a generalization of the concept of an \emph{effective procedure}~\cite[Sect.\ 1.7]{Enderton2001}.  This generalization can be obtained by allowing, in addition to the usual data processing instructions of an effective procedure, instructions for interacting with the physical universe, where any interaction is treated as a nondeterministic oracle.\footnote{
Oracles and nondeterministic computations are described in theoretical computer science textbooks, such as the textbook by Davis et al.~\cite{Davis1994}.
}
%
%
In the spirit of the concept of an effective procedure, we place no bounds on the time, resources, or preparations that might be required to perform a measuring operation.  A measuring operation may be a simple act, such as listening for a particular sound with one's ears, and recording a $1$ or $0$ to signify whether or not the sound was heard.  Alternatively, a measuring operation may require elaborately constructed measuring instruments.  A measuring operation might extend for a period of time that is much longer than the duration of the phenomenon being observed, especially if the operation requires time-consuming preparations or a lengthy mathematical analysis of data that has been collected.  If an error analysis of the collected data is part of the process of producing the measurement result, then that error analysis is necessarily part of the measuring operation.\footnote{
Error bounds on a data point can be encoded as a nonnegative integer, as described in Section~\ref{S:real}.
}  And a nonnegative integer physical model, even if it is faithful, does not guarantee that the measuring operations can always be performed, or performed to completion.  Indeed, a measuring operation might require more resources than are available in the entire universe.  But Definition~\ref{D:faithful} does guarantee that for a faithful nonnegative integer physical model with structure $\mathfrak{A}$, if a measuring operation $\mathrm{op}(f)$ can be performed to completion, then the result of that measurement will be equal to $f^\mathfrak{A}(s)$ for some state $s\in\lvert\mathfrak{A}\rvert_S$.

We use Cantor's pairing function $J(a,b)=\tfrac{1}{2}\bigl((a+b)^2+3a+b\bigr)$ to encode any ordered pair $\langle a,b\rangle$ of nonnegative integers as a single nonnegative integer.  We write $J(a,b,c)$ as an abbreviation for $J\bigl(J(a,b),c\bigr)$, to encode ordered triples of nonnegative integers.  We also write $J(a,b,c,d)$ as an abbreviation for $J\bigl(J\bigl(J(a,b),c\bigr),d\bigr)$, to encode ordered quadruples of nonnegative integers, and so on.  Any measuring operation that encodes the results of two or more measuring operations in this manner is said to be a \emph{joint measuring operation}.  For example, consider the following nonnegative integer physical model for the motion of a projectile fired from a cannon at $5$ meters per second in an inertial reference frame, and in the absence of any external forces, such as gravity or air resistance.
\begin{model}\label{M:cannon}
Consider a two-sorted nonnegative integer physical language with a symbol $f$ for an observable quantity, and with no additional physical symbols.  Let $\mathfrak{A}$ be a structure for this language, where $\lvert\mathfrak{A}\rvert_S$ is the set of all $t\in\mathbb{N}$.  Define $\mathrm{op}(f)$ to be an operation that measures the nonnegative integer number of seconds $s$ since the projectile was fired, together with the number of meters $m$ between the cannon and the projectile at that time.\footnote{
We require the cannon to be at rest in the inertial reference frame, and for the time and distance measurements to be made relative to this frame.  We also require the distance measurement to be made within $\pm\Delta s$ seconds of the nonnegative integer number of seconds, and to be accurate to within $\pm\Delta m$ meters, where $5\lvert\Delta s\rvert+\lvert\Delta m\rvert<0.5$.
%
%
The resulting distance is then rounded to the nearest nonnegative integer.
}  The result of this joint measuring operation is encoded as $J(s,m)$.  Let $f^\mathfrak{A}(t)=J(t,5t)$.
\end{model}
\noindent
If one asserts that this model is faithful, then one asserts that for each measurement $J(s,m)$, there exists a $t\in\mathbb{N}$ such that $J(s,m)=J(t,5t)$.

More complicated examples of nonnegative integer physical models are described in subsequent sections.  Notable examples include a model for the pressure, volume, and temperature of one mole of a gas (Model~\ref{M:ideal-gas}), a model for lower and upper bounds on the number of molecules in a sample of a chemical compound (Model~\ref{M:molecules-bounds}), and a model for $\beta^-$ decay in a sample of copper-64 (Model~\ref{M:decay}).

\section{Properties of Nonnegative Integer Physical Models}\label{S:properties}

Given a nonnegative integer physical model, one of the most central questions is whether that model is faithful.  But there are other questions that can also be asked about nonnegative integer physical models.  For example, given a structure $\mathfrak{A}$ for a nonnegative integer physical language where the nonlogical symbols of first-order number theory are assigned their traditional set-theoretic meanings, one can ask whether there exists \emph{any faithful model} that has the given structure $\mathfrak{A}$.  The following theorem shows that this question has a trivial answer.

\begin{theorem}
Let $\mathfrak{A}$ be any structure for a nonnegative integer physical language where the nonlogical symbols of first-order number theory are assigned their traditional set-theoretic meanings.  Then there exists a faithful nonnegative integer physical model $\mathcal{A}$ that has the structure $\mathfrak{A}$.
\end{theorem}
\begin{proof}
Define $\mathcal{A}$ to be a nonnegative integer physical model that has the structure $\mathfrak{A}$, and such that, for each symbol $f$ for an observable quantity, $\mathrm{op}(f)$ is a physical measuring operation that always fails and can never be completed.\footnote{
This measuring operation is analogous to an effective procedure for the computable partial function whose domain is empty~\cite[p.\ 252]{Enderton2001}.  That is, it is analogous to a computer program that always aborts, or that always goes into an infinite loop.
}  Then it is vacuously true,
%
%
by Definition~\ref{D:faithful}, that $\mathcal{A}$ is faithful.
\end{proof}

Various relations between structures are studied in mathematical logic.  For example, a structure $\mathfrak{B}$ might be a reduct of, an extension of, or isomorphic to a structure $\mathfrak{A}$.   In many-sorted first-order logic, these relations can be defined as follows.
%
%
A structure $\mathfrak{B}$ is said to be a \emph{reduct} of a structure $\mathfrak{A}$ if and only if
\begin{enumerate}
\renewcommand{\theenumi}{\roman{enumi}}
\item each nonlogical symbol and each equality symbol in the language of $\mathfrak{B}$ is a nonlogical symbol or equality symbol, respectively, in the language of $\mathfrak{A}$; and
\item in the structure $\mathfrak{B}$, each nonlogical symbol is assigned the same set-theoretic object that it is assigned in the structure $\mathfrak{A}$.
\end{enumerate}
A structure $\mathfrak{B}$ is said to be an \emph{extension} of a structure $\mathfrak{A}$ if and only if
\begin{enumerate}
\renewcommand{\theenumi}{\roman{enumi}}
\item $\mathfrak{A}$ and $\mathfrak{B}$ have the same language;
\item for each sort $i$, $\lvert\mathfrak{A}\rvert_i\subseteq\lvert\mathfrak{B}\rvert_i$;
\item for each $n$-place predicate symbol $P$ of sort $\langle i_1,\ldots,i_n\rangle$, $P^\mathfrak{A}$ is the restriction of $P^\mathfrak{B}$ to $\lvert\mathfrak{A}\rvert_{i_1}\times\cdots\times\lvert\mathfrak{A}\rvert_{i_n}$;
\item for each $n$-place function symbol $f$ of sort $\langle i_1,\ldots,i_n,i_{n+1}\rangle$, $f^\mathfrak{A}$ is the restriction of $f^\mathfrak{B}$ to $\lvert\mathfrak{A}\rvert_{i_1}\times\cdots\times\lvert\mathfrak{A}\rvert_{i_n}$; and
\item for each constant symbol $c$, $c^\mathfrak{A}=c^\mathfrak{B}$.
\end{enumerate}
And we say that a structure $\mathfrak{B}$ is \emph{isomorphic} to a structure $\mathfrak{A}$ if and only if
\begin{enumerate}
\renewcommand{\theenumi}{\roman{enumi}}
\item $\mathfrak{A}$ and $\mathfrak{B}$ have the same language;
\item for each sort $i$ there is a one-to-one correspondence $h_i$ from $\lvert\mathfrak{B}\rvert_i$ onto $\lvert\mathfrak{A}\rvert_i$;
\item for each $n$-place predicate symbol $P$ of sort $\langle i_1,\ldots,i_n\rangle$, and for each $\langle b_1,\ldots,\linebreak[0]b_n\rangle$
%
%
in $\lvert\mathfrak{B}\rvert_{i_1}\times\cdots\times\lvert\mathfrak{B}\rvert_{i_n}$, $\langle b_1,\ldots,b_n\rangle\in P^\mathfrak{B}$ if and only if $\bigl\langle h_{i_1}(b_1),\ldots,\linebreak[0]h_{i_n}(b_n)\bigr\rangle\in P^\mathfrak{A}$;
%
%
\item for each $n$-place function symbol $f$ of sort $\langle i_1,\ldots,i_n,i_{n+1}\rangle$, and for each $\langle b_1,\ldots,b_n\rangle$ in $\lvert\mathfrak{B}\rvert_{i_1}\times\cdots\times\lvert\mathfrak{B}\rvert_{i_n}$, $h_{i_{n+1}}\bigl(f^\mathfrak{B}(b_1,\ldots,b_n)\bigr)=f^\mathfrak{A}\bigl(h_{i_1}(b_1),\ldots,\linebreak[0]h_{i_n}(b_n)\bigr)$; and
%
%
\item for each constant symbol $c$ of sort $i$, $h_i(c^\mathfrak{B})=c^\mathfrak{A}$.
\end{enumerate}
Note that in the special case when $\mathfrak{A}$ and $\mathfrak{B}$ are structures for nonnegative integer physical models, it can be shown~\cite[Chap.\ 2, Thm.\ 3.4]{Auslander1974} that if $\mathfrak{A}$ and $\mathfrak{B}$ are isomorphic, then $h_N$ is necessarily the identity function.

Now, these relations can be generalized in the following manner.
\begin{definition}\label{D:isomorphic}
Let $\mathcal{A}$ be a nonnegative integer physical model with a structure $\mathfrak{A}$, and let $\mathcal{B}$ be a nonnegative integer physical model with a structure $\mathfrak{B}$.  We say that the model $\mathcal{B}$ is a \emph{reduct} of, an \emph{extension} of, or \emph{isomorphic} to the model $\mathcal{A}$ if and only if
\begin{enumerate}
\renewcommand{\theenumi}{\roman{enumi}}
\item the structure $\mathfrak{B}$ is a reduct of, an extension of, or isomorphic to the structure $\mathfrak{A}$, respectively; and
\item for each symbol $f$ for an observable quantity in the language of $\mathfrak{B}$, $f$ is assigned the same measuring operation in both models.
\end{enumerate}
In addition, we say that $\mathcal{A}$ is an \emph{expansion} of $\mathcal{B}$ if and only if $\mathcal{B}$ is a reduct of $\mathcal{A}$.  And we say that $\mathcal{A}$ is a \emph{submodel} of $\mathcal{B}$ if and only if $\mathcal{B}$ is an extension of $\mathcal{A}$.
\end{definition}

This definition ensures that if $\mathcal{A}$ and $\mathcal{B}$ are isomorphic nonnegative integer physical models, then $\mathcal{A}$ is faithful if and only if $\mathcal{B}$ is faithful.
%
%
We also have the following corollaries of Definition~\ref{D:isomorphic}.
\begin{corollary}\label{C:reduct}
Let $\mathcal{A}$ and $\mathcal{B}$ be any nonnegative integer physical models such that $\mathcal{B}$ is a reduct of $\mathcal{A}$.  If $\mathcal{A}$ is faithful, then $\mathcal{B}$ is faithful.
\end{corollary}
\begin{proof}
Let $\mathfrak{A}$ be the structure of $\mathcal{A}$, and let $\mathfrak{B}$ be the structure of $\mathcal{B}$.  Suppose that $\mathcal{A}$ is faithful.  Now consider any symbol $f$ for an observable quantity of $\mathcal{B}$, consider any nonnegative integer $n$, and suppose that a measurement result of $\mathrm{op}(f)$ is equal to $n$.  Because $\mathcal{B}$ is a reduct of $\mathcal{A}$, $f$ is also a symbol for an observable quantity of $\mathcal{A}$.  And since $\mathcal{A}$ is faithful, there exists an $s\in\lvert\mathfrak{A}\rvert_S=\lvert\mathfrak{B}\rvert_S$ such that $n=f^\mathfrak{A}(s)=f^\mathfrak{B}(s)$.  By Definition~\ref{D:faithful}, $\mathcal{B}$ is faithful.
\end{proof}

\begin{corollary}\label{C:one-symbol}
Let $\mathcal{A}$ be any nonnegative integer physical model, and let $B$ be the set of all reducts of $\mathcal{A}$ that are nonnegative integer physical models with exactly one symbol for an observable quantity.  Then, $\mathcal{A}$ is faithful if and only if every member of the set $B$ is faithful.
\end{corollary}
%
%
\begin{proof}
If $\mathcal{A}$ is faithful, then by Corollary~\ref{C:reduct}, every member of $B$ is faithful.  Conversely, suppose that every member of $B$ is faithful.  Now consider any symbol $f$ for an observable quantity of $\mathcal{A}$, consider any nonnegative integer $n$, and suppose that a measurement result of $\mathrm{op}(f)$ is equal to $n$.  Let $\mathcal{B}$ be a member of $B$ that has $f$ as its only symbol for an observable quantity, and let $\mathfrak{A}$ and $\mathfrak{B}$ be the structures of $\mathcal{A}$ and $\mathcal{B}$, respectively.  Because every member of $B$ is faithful, there exists an $s\in\lvert\mathfrak{B}\rvert_S=\lvert\mathfrak{A}\rvert_S$ such that $n=f^\mathfrak{B}(s)=f^\mathfrak{A}(s)$.  By Definition~\ref{D:faithful}, $\mathcal{A}$ is faithful.
\end{proof}

\begin{corollary}\label{C:extension}
Let $\mathcal{A}$ and $\mathcal{B}$ be any nonnegative integer physical models such that $\mathcal{B}$ is an extension of $\mathcal{A}$.  If $\mathcal{A}$ is faithful, then $\mathcal{B}$ is faithful.
\end{corollary}
\begin{proof}
Let $\mathfrak{A}$ be the structure of $\mathcal{A}$, and let $\mathfrak{B}$ be the structure of $\mathcal{B}$.  Suppose that $\mathcal{A}$ is faithful.  Now consider any symbol $f$ for an observable quantity of $\mathcal{B}$, consider any nonnegative integer $n$, and suppose that a measurement result of $\mathrm{op}(f)$ is equal to $n$.  Because $\mathcal{A}$ is faithful, there exists an $s\in\lvert\mathfrak{A}\rvert_S$ such that $f^\mathfrak{A}(s)=n$.  But $\mathcal{B}$ is an extension of $\mathcal{A}$, so $s\in\lvert\mathfrak{A}\rvert_S\subseteq\lvert\mathfrak{B}\rvert_S$ and $n=f^\mathfrak{A}(s)=f^\mathfrak{B}(s)$.  By Definition~\ref{D:faithful}, $\mathcal{B}$ is faithful.
\end{proof}

Hence, every extension of a faithful nonnegative integer physical model is itself faithful.  But in a certain sense, an extension of a model is also weaker than the original model.  This notion of the relative strength or weakness of a nonnegative integer physical model is formalized in the following manner.

Consider any nonnegative integer physical models $\mathcal{A}$ and $\mathcal{B}$ with structures $\mathfrak{A}$ and $\mathfrak{B}$, respectively.  Suppose that both models have the same language and measuring operations,
%
%
and consider any symbol $f$ for an observable quantity.  Let $\mathrm{ran}(\,f^\mathfrak{A})$ denote the range of $f^\mathfrak{A}$, and similarly for $f^\mathfrak{B}$.  That is,
\begin{equation*}
\mathrm{ran}(\,f^\mathfrak{A})=\bigl\{\,f^\mathfrak{A}(s)\,\bigm\vert\,s\in\lvert\mathfrak{A}\rvert_S\,\bigr\}.
\end{equation*}
Note that if $\mathcal{A}$ is faithful, then every measurement result for $\mathrm{op}(f)$ is a member of $\mathrm{ran}(\,f^\mathfrak{A})$.  For this reason, we regard $\mathrm{ran}(\,f^\mathfrak{A})$ as the set of possible measurement results for $\mathrm{op}(f)$ that are \emph{allowed} by the model $\mathcal{A}$.  And if $\mathrm{ran}(\,f^\mathfrak{A})\subseteq\mathrm{ran}(\,f^\mathfrak{B})$, then $\mathcal{A}$ might allow fewer possible values for $\mathrm{op}(f)$ than the model $\mathcal{B}$ allows.  In this sense, the observable quantity $f^\mathfrak{A}$ is \emph{stronger} than the observable quantity $f^\mathfrak{B}$.  A similar notion of relative strength was used by Popper.  In particular, Popper~\cite[Sect.\ 20]{Popper1959} identified the strength of a theory with its ``degree of falsifiability'', stating that a theory is strengthened if it ``now rules out more than it did previously: it prohibits more.''  In a similar way, we say that the model $\mathcal{A}$ is \emph{stronger} than the model $\mathcal{B}$ if and only if, for each symbol $f$ for an observable quantity, $\mathrm{ran}(\,f^\mathfrak{A})\subseteq\mathrm{ran}(\,f^\mathfrak{B})$.
%
%
This notion is also expressed by saying that $\mathcal{B}$ is \emph{weaker} than $\mathcal{A}$.

Because an extension of a nonnegative integer physical model is always weaker than the original model, Corollary~\ref{C:extension} is a special case of the following, more general corollary.
\begin{corollary}
Let $\mathcal{A}$ and $\mathcal{B}$ be nonnegative integer physical models that have the same language and measuring operations, and let $\mathcal{A}$ be stronger than $\mathcal{B}$.  If $\mathcal{A}$ is faithful, then $\mathcal{B}$ is faithful.
\end{corollary}
\begin{proof}
Let $\mathfrak{A}$ be the structure of $\mathcal{A}$, and let $\mathfrak{B}$ be the structure of $\mathcal{B}$.  Suppose that $\mathcal{A}$ is faithful.  Now consider any symbol $f$ for an observable quantity, consider any nonnegative integer $n$, and suppose that a measurement result of $\mathrm{op}(f)$ is equal to $n$.  Because $\mathcal{A}$ is faithful, $n\in\mathrm{ran}(\,f^\mathfrak{A})\subseteq\mathrm{ran}(\,f^\mathfrak{B})$.  Hence, there exists an $s\in\lvert\mathfrak{B}\rvert_S$ such that $f^\mathfrak{B}(s)=n$.  By Definition~\ref{D:faithful}, $\mathcal{B}$ is faithful.
\end{proof}

We define $\mathcal{A}$ to be \emph{observationally equivalent} to $\mathcal{B}$ if and only if both $\mathcal{A}$ is stronger than $\mathcal{B}$, and $\mathcal{B}$ is stronger than $\mathcal{A}$.  That is, $\mathcal{A}$ is observationally equivalent to $\mathcal{B}$ if and only if, for each symbol $f$ for an observable quantity, $\mathrm{ran}(\,f^\mathfrak{A})=\mathrm{ran}(\,f^\mathfrak{B})$.  If two nonnegative integer physical models are isomorphic, then they are also observationally equivalent.
%
%
In addition, we have the following corollary of the definition of observational equivalence.
\begin{corollary}\label{C:countable}
Let $\mathcal{A}$ be any nonnegative integer physical model.  Then, there exists a nonnegative integer physical model $\mathcal{B}$ with a structure $\mathfrak{B}$ such that $\lvert\mathfrak{B}\rvert_S=\mathbb{N}$, and such that $\mathcal{A}$ is observationally equivalent to $\mathcal{B}$.
\end{corollary}
\begin{proof}
Let $\mathcal{B}$ be a nonnegative integer physical model that has the same language and measuring operations as $\mathcal{A}$, and that has a structure $\mathfrak{B}$ which is defined so that $\lvert\mathfrak{B}\rvert_S=\mathbb{N}$.  Let $\mathfrak{A}$ be the structure of $\mathcal{A}$, and consider any symbol $f$ for an observable quantity.  Because $\mathrm{ran}(\,f^\mathfrak{A})\subseteq\mathbb{N}$ is countable, there is a function with domain $\mathbb{N}$ and range $\mathrm{ran}(\,f^\mathfrak{A})$.  Define $f^\mathfrak{B}$ to be this function.  Because $\mathrm{ran}(\,f^\mathfrak{A})=\mathrm{ran}(\,f^\mathfrak{B})$ for each symbol $f$ for an observable quantity, $\mathcal{A}$ is observationally equivalent to $\mathcal{B}$.
\end{proof}

\section{Computable Physical Models}\label{S:cp-models}

Given a faithful nonnegative integer physical model with a structure $\mathfrak{A}$ and a symbol $f$ for an observable quantity, each measurement result of $\mathrm{op}(f)$ is necessarily a member of $\mathrm{ran}(\,f^\mathfrak{A})$.  But $\mathrm{ran}(\,f^\mathfrak{A})$ might contain additional values that are not measurement results for $\mathrm{op}(f)$.  The set that contains \emph{exactly} those nonnegative integers which are measurement results for $\mathrm{op}(f)$ is denoted $O_f$.  Then, we say that a nonnegative integer physical model with a structure $\mathfrak{A}$ is \emph{maximally faithful} if and only if, for each symbol $f$ for an observable quantity, $\mathrm{ran}(\,f^\mathfrak{A})=O_f$.  One consequence of this definition is that if a model is maximally faithful, then $O_f$ cannot be empty for any symbol $f$ for an observable quantity.  This is because the definition of a structure~\cite[Sect.\ 4.3]{Enderton2001} requires $\mathrm{ran}(\,f^\mathfrak{A})$ to be nonempty.

Another way to characterize the maximally faithful nonnegative integer physical models is given by the following theorem.
\begin{theorem}\label{T:max-strong}
A nonnegative integer physical model $\mathcal{A}$ is maximally faithful if and only if
\begin{enumerate}
\renewcommand{\theenumi}{\roman{enumi}}
\item\label{P:faithful} $\mathcal{A}$ is faithful; and
\item\label{P:weaker} every faithful nonnegative integer physical model that has the same language and measuring operations as $\mathcal{A}$ is weaker than $\mathcal{A}$.
\end{enumerate}
\end{theorem}
\begin{proof}
Consider any nonnegative integer physical model $\mathcal{A}$ with a structure $\mathfrak{A}$, and suppose that $\mathcal{A}$ is maximally faithful.  Then consider any symbol $f$ for an observable quantity of $\mathcal{A}$, consider any nonnegative integer $n$, and suppose that a measurement result of $\mathrm{op}(f)$ is equal to $n$.  Because $\mathcal{A}$ is maximally faithful, $n\in O_f=\mathrm{ran}(\,f^\mathfrak{A})$.  Hence, there exists an $s\in\lvert\mathfrak{A}\rvert_S$ such that $f^\mathfrak{A}(s)=n$.  By Definition~\ref{D:faithful}, $\mathcal{A}$ is faithful.  Now consider any faithful nonnegative integer physical model $\mathcal{B}$ that has the same language and measuring operations as $\mathcal{A}$.  Let $\mathfrak{B}$ be the structure of $\mathcal{B}$.  Because $\mathcal{B}$ is faithful, $\mathrm{ran}(\,f^\mathfrak{A})=O_f\subseteq\mathrm{ran}(\,f^\mathfrak{B})$.  Thus, $\mathcal{B}$ is weaker than $\mathcal{A}$.  We have shown that conditions~(\ref{P:faithful}) and~(\ref{P:weaker}) hold if $\mathcal{A}$ is maximally faithful.

Conversely, suppose that conditions~(\ref{P:faithful}) and~(\ref{P:weaker}) hold.  Let $\mathcal{B}$ be a nonnegative integer physical model that has the same language and measuring operations as $\mathcal{A}$, and that has a structure $\mathfrak{B}$ which is defined so that $\lvert\mathfrak{B}\rvert_S=\mathbb{N}$, and so that
\begin{equation*}
\mathrm{ran}(\,f^\mathfrak{B})=\begin{cases}
O_f &\text{if $O_f$ is nonempty}\\
\mathbb{N}-\{a_f\} &\text{otherwise}
\end{cases}
\end{equation*}
for each symbol $f$ for an observable quantity, where $a_f$ denotes the smallest nonnegative integer in $\mathrm{ran}(\,f^\mathfrak{A})$, and where $\mathbb{N}-\{a_f\}$ denotes the complement of $\{a_f\}$.  By Definition~\ref{D:faithful}, $\mathcal{B}$ is faithful.  Hence, by condition~(\ref{P:weaker}),
\begin{equation*}
\mathrm{ran}(\,f^\mathfrak{A})\subseteq\mathrm{ran}(\,f^\mathfrak{B})
\end{equation*}
for each symbol $f$ for an observable quantity.  But if $O_f$ is empty, then
\begin{equation*} a_f\in\mathrm{ran}(\,f^\mathfrak{A})\subseteq\mathrm{ran}(\,f^\mathfrak{B})=\mathbb{N}-\{a_f\}.
\end{equation*}
This is impossible, since $a_f\notin\mathbb{N}-\{a_f\}$.  Hence, it must be the case that $O_f$ is nonempty.  It then follows that $\mathrm{ran}(\,f^\mathfrak{A})\subseteq\mathrm{ran}(\,f^\mathfrak{B})=O_f$ for each symbol $f$ for an observable quantity.  And by condition~(\ref{P:faithful}), $O_f\subseteq\mathrm{ran}(\,f^\mathfrak{A})$.  Therefore, $\mathrm{ran}(\,f^\mathfrak{A})=O_f$ for each symbol $f$ for an observable quantity.  We have shown that if conditions~(\ref{P:faithful}) and~(\ref{P:weaker}) hold, then $\mathcal{A}$ is maximally faithful.
\end{proof}

Now, the existence~\cite{PourEl1981} of a noncomputable weak solution to the wave equation, with computable initial conditions, is a well-known example of a noncomputability result in mathematics.  But it is an open question~\cite[pp.\ 330--331]{Weihrauch2002} whether this sort of noncomputability can be observed in the physical universe.  One way to formalize this question is to ask whether or not the following hypothesis is true.

\begin{cuh}
For every physical measuring operation, the set that contains exactly those nonnegative integers which are measurement results for the operation is a recursively enumerable set.
\end{cuh}

That is, the computable universe hypothesis states that for each nonnegative integer physical model $\mathcal{A}$, and for each symbol $f$ for an observable quantity in the language of $\mathcal{A}$, the set $O_f$ is recursively enumerable.\footnote{
Statements such as this are sometimes called the \emph{physical form} of the Church-Turing thesis.  See Rosen~\cite[p.\ 377]{Rosen1962}, for example.  But to avoid confusion with the Church-Turing thesis, which is a distinct hypothesis~\cite[Sect.\ 1]{Gandy1980}, we refrain from using that terminology.
}  In the context of this hypothesis, it is natural to consider nonnegative integer physical models of the following form.
\begin{definition}
A nonnegative integer physical model with a structure $\mathfrak{A}$ is said to be a \emph{computable physical model}\footnote{
In previous publications~\cite{Szudzik2012,Szudzik2013}, the definition of a computable physical model was slightly different from the definition given here.  In those previous publications, $\lvert\mathfrak{A}\rvert_S$ was required to be a recursive set (in other words, a computable set), and for each symbol $f$ for an observable quantity, $f^\mathfrak{A}$ was required to be a recursive total function (in other words, a computable function) whose domain is restricted to $\lvert\mathfrak{A}\rvert_S$.  But the two different definitions are equivalent, in the sense that if a nonnegative integer physical model is computable according to either definition, then it is isomorphic to a model that is computable according to the other definition.  See Szudzik~\cite[Thm.\ 10.5]{Szudzik2013}.
%
%
} if and only if
\begin{enumerate}
\renewcommand{\theenumi}{\roman{enumi}}
\item $\lvert\mathfrak{A}\rvert_S$ is a recursively enumerable set of nonnegative integers; and
\item for each symbol $f$ for an observable quantity, $f^\mathfrak{A}$ is a recursive partial function that has $\lvert\mathfrak{A}\rvert_S$ as its domain.
\end{enumerate}
\end{definition}
%
%
\noindent
An immediate corollary of this definition is that every nonnegative integer physical model that is a reduct of a computable physical model is itself a computable physical model.  Another corollary is that for any computable physical model with a structure $\mathfrak{A}$, and with a symbol $f$ for an observable quantity, $\mathrm{ran}(\,f^\mathfrak{A})$ is a nonempty recursively enumerable set.  Note that Models~\ref{M:baryon} and~\ref{M:cannon} are examples of computable physical models.  And Corollary~\ref{C:countable} can be adapted to computable physical models in the following manner.
\begin{corollary}
Let $\mathcal{A}$ be any computable physical model.  Then, there exists a computable physical model $\mathcal{B}$ with a structure $\mathfrak{B}$ such that $\lvert\mathfrak{B}\rvert_S=\mathbb{N}$, and such that $\mathcal{A}$ is observationally equivalent to $\mathcal{B}$.
\end{corollary}
\begin{proof}
Let $\mathcal{B}$ be a nonnegative integer physical model that has the same language and measuring operations as $\mathcal{A}$, and that has a structure $\mathfrak{B}$ which is defined so that $\lvert\mathfrak{B}\rvert_S=\mathbb{N}$.  Let $\mathfrak{A}$ be the structure of $\mathcal{A}$, and consider any symbol $f$ for an observable quantity.  Because $\mathcal{A}$ is a computable physical model, $\mathrm{ran}(\,f^\mathfrak{A})$ is a nonempty recursively enumerable set.  Therefore, there exists~\cite[p.\ 82, Thm.\ 4.9]{Davis1994}
%
%
a recursive partial function with domain $\mathbb{N}$ and range $\mathrm{ran}(\,f^\mathfrak{A})$.  Define $f^\mathfrak{B}$ to be this function.  Then $\mathcal{B}$ is a computable physical model.  And because $\mathrm{ran}(\,f^\mathfrak{A})=\mathrm{ran}(\,f^\mathfrak{B})$ for each symbol $f$ for an observable quantity, $\mathcal{A}$ is observationally equivalent to $\mathcal{B}$.
\end{proof}

Alternate characterizations of the computable universe hypothesis are provided by the following theorems.
\begin{theorem}\label{T:cuh}
The computable universe hypothesis is true if and only if every maximally faithful nonnegative integer physical model is observationally equivalent to a computable physical model.
\end{theorem}
%
%
\begin{proof}
Suppose that the computable universe hypothesis is true, and consider any maximally faithful nonnegative integer physical model $\mathcal{A}$ with a structure $\mathfrak{A}$.  Let $\mathcal{B}$ be a nonnegative integer physical model that has the same language and measuring operations as $\mathcal{A}$.  Define the structure $\mathfrak{B}$ of $\mathcal{B}$ so that $\lvert\mathfrak{B}\rvert_S=\mathbb{N}$.  Because $\mathcal{A}$ is maximally faithful, $\mathrm{ran}(\,f^\mathfrak{A})=O_f$ for each symbol $f$ for an observable quantity.  Thus, $O_f$ is nonempty.  And because we are assuming the computable universe hypothesis, $O_f$ is recursively enumerable.  Since $O_f$ is a nonempty recursively enumerable set, there must exist a recursive partial function with domain $\mathbb{N}$ and range $O_f$.  Define $f^\mathfrak{B}$ to be this function.  Then $\mathcal{B}$ is a computable physical model that is observationally equivalent to $\mathcal{A}$ because $\mathrm{ran}(\,f^\mathfrak{A})=O_f=\mathrm{ran}(\,f^\mathfrak{B})$ for each symbol $f$ for an observable quantity.  We have shown that if the computable universe hypothesis is true, then every maximally faithful nonnegative integer physical model $\mathcal{A}$ is observationally equivalent to a computable physical model $\mathcal{B}$.

Alternatively, suppose that the computable universe hypothesis is false.  Then there exists a physical measuring operation such that the set containing exactly those nonnegative integers which are measurement results for the operation is not a recursively enumerable set.  This set is necessarily nonempty, since the empty set is recursively enumerable.  Now let $\mathcal{A}$ be a nonnegative integer physical model with a structure $\mathfrak{A}$, and with $f$ as the only symbol for an observable quantity.  Define $\mathrm{op}(f)$ to be the aforementioned measuring operation, and define $f^\mathfrak{A}$ so that $\mathrm{ran}(\,f^\mathfrak{A})=O_f$.  By definition, $\mathcal{A}$ is maximally faithful.  Next, consider any nonnegative integer physical model $\mathcal{B}$ that is observationally equivalent to $\mathcal{A}$, and let $\mathfrak{B}$ be the structure of $\mathcal{B}$.  By the definition of observational equivalence, $\mathrm{ran}(\,f^\mathfrak{B})=\mathrm{ran}(\,f^\mathfrak{A})=O_f$.  Because this is not a recursively enumerable set, $\mathcal{B}$ cannot be a computable physical model.  We have shown that if the computable universe hypothesis is false, then there is a maximally faithful nonnegative integer physical model $\mathcal{A}$ that is not observationally equivalent to any computable physical model $\mathcal{B}$.
\end{proof}

\begin{theorem}
The computable universe hypothesis is true if and only if, for each faithful nonnegative integer physical model $\mathcal{A}$, there is a faithful computable physical model that has the same language and measuring operations as  $\mathcal{A}$, and that is stronger than $\mathcal{A}$.
\end{theorem}
\begin{proof}
Suppose that the computable universe hypothesis is true, and consider any faithful nonnegative integer physical model $\mathcal{A}$ with a structure $\mathfrak{A}$.  Because $\mathcal{A}$ is faithful, $O_f\subseteq\mathrm{ran}(\,f^\mathfrak{A})$ for each symbol $f$ for an observable quantity.  Now let $\mathcal{B}$ be a nonnegative integer physical model that has the same language and measuring operations as $\mathcal{A}$.  Define the structure $\mathfrak{B}$ of $\mathcal{B}$ so that $\lvert\mathfrak{B}\rvert_S=\mathbb{N}$.  For each symbol $f$ for an observable quantity, the computable universe hypothesis implies that $O_f$ is a recursively enumerable set.  There are two cases to consider.
\begin{description}
\item[Case 1] If $O_f$ is nonempty, then there exists a recursive partial function with domain $\mathbb{N}$ and range $O_f$.  Define $f^\mathfrak{B}$ to be this function.  Note that $\mathrm{ran}(\,f^\mathfrak{B})=O_f\subseteq\mathrm{ran}(\,f^\mathfrak{A})$.
\item[Case 2] If $O_f$ is the empty set then, for each $s\in\mathbb{N}$, define $f^\mathfrak{B}(s)=a_f$, where $a_f$ denotes the smallest nonnegative integer in $\mathrm{ran}(\,f^\mathfrak{A})$.  Note that $\mathrm{ran}(\,f^\mathfrak{B})=\{a_f\}\subseteq\mathrm{ran}(\,f^\mathfrak{A})$.
\end{description}
In either case, $f^\mathfrak{B}$ is a recursive partial function with domain $\mathbb{N}$, and $O_f\subseteq\mathrm{ran}(\,f^\mathfrak{B})\linebreak[0]\subseteq\mathrm{ran}(\,f^\mathfrak{A})$.
%
%
By definition, $\mathcal{B}$ is a faithful computable physical model that is stronger than $\mathcal{A}$.

Conversely, suppose that for each faithful nonnegative integer physical model $\mathcal{A}$, there is a faithful computable physical model that has the same language and measuring operations as $\mathcal{A}$, and that is stronger than $\mathcal{A}$.  Then for each maximally faithful nonnegative integer physical model $\mathcal{A}$, there is a faithful computable physical model $\mathcal{B}$ that has the same language and measuring operations as  $\mathcal{A}$, and that is stronger than $\mathcal{A}$.  But by  condition~(\ref{P:weaker}) of Theorem~\ref{T:max-strong}, $\mathcal{B}$ is also weaker than $\mathcal{A}$.  Hence, $\mathcal{B}$ is observationally equivalent to $\mathcal{A}$.  We have shown that every maximally faithful nonnegative integer physical model $\mathcal{A}$ is observationally equivalent to a computable physical model $\mathcal{B}$.  Therefore, by Theorem~\ref{T:cuh}, the computable universe hypothesis is true.
\end{proof}

Now consider any two-sorted nonnegative integer physical language that has symbols for observable quantities, a predicate symbol for equality of sort $\langle S,S\rangle$, a function symbol $\mathrm{conv}_{S\to N}$ of sort $\langle S,N\rangle$, and no additional physical symbols.  Let $\mathfrak{A}$ be the structure of a computable physical model that has this language.  Letting $\mathrm{conv}_{S\to N}^\mathfrak{A}(n)=n$ for each nonnegative integer $n\in\lvert\mathfrak{A}\rvert_S$, $\mathfrak{A}$ satisfies
%
%
\begin{equation}\label{E:one-to-one}
\forall_S\,s\;\forall_S\,t\,\bigl(\,\mathrm{conv}_{S\to N}(s)=\mathrm{conv}_{S\to N}(t)\;\rightarrow\;s=t\,\bigr)
\end{equation}
and
\begin{equation}\label{E:domain}
\forall_N\,x\,\Bigl(\,\phi(x)\;\leftrightarrow\;\exists_S\,s\,\bigl(\,\mathrm{conv}_{S\to N}(s)=x\,\bigr)\,\Bigr),
\end{equation}
where $\phi(x)$ is a formula that defines the set $\lvert\mathfrak{A}\rvert_S$ within the standard model $\mathfrak{N}$ of first-order number theory.  And for each symbol $f$ for an observable quantity, $\mathfrak{A}$ satisfies
\begin{equation}\label{E:function}
\forall_N\,x\,\forall_N\,y\,\Bigl(\,\psi_f(x,y)\;\leftrightarrow\;\exists_S\,s\,\bigl(\,\mathrm{conv}_{S\to N}(s)=x\;\wedge\;f(s)=y\,\bigr)\,\Bigr),
\end{equation}
where $\psi_f(x,y)$ is a formula that defines $f^\mathfrak{A}$ as a relation in $\mathfrak{N}$.  Taken together, we can regard formulas~\eqref{E:one-to-one} and~\eqref{E:domain}, together with a formula of the form~\eqref{E:function} for each symbol $f$ for an observable quantity, as axioms that extend first-order number theory.
%
%
Note that every $\omega$-model\footnote{
An $\omega$-model of the axioms is any structure that satisfies the axioms and that assigns the nonlogical symbols of first-order number theory their traditionally intended meanings.
%
%
%
%
See Enderton~\cite[p.\ 304]{Enderton2001} and Barwise~\cite[p.\ 42]{Barwise1977}.
%
%
} of these axioms is isomorphic to the structure $\mathfrak{A}$ of the computable physical model.

\section{Definitional Expansions}\label{S:definitional}

Many axiom systems have languages with a small number of symbols.  The axioms of set theory, for example, are often written in a language where $=$ and $\in$ are the only predicate symbols.  A structure $\mathfrak{A}$ for set theory is a structure for this language.  Other predicate symbols that are commonly used by set theorists, such as the subset symbol ($\subseteq$) and proper subset symbol ($\subsetneq$), are usually defined in terms of $=$ and $\in$.  But the structure $\mathfrak{A}$ can be expanded to incorporate these defined symbols into its language.  This expanded structure is known as a \emph{definitional expansion}\footnote{
Definitional expansions are described in greater detail by Hodges~\cite[pp.\ 59--60]{Hodges1993}.
} of $\mathfrak{A}$.

Defined symbols for observable quantities can also be introduced to the language of a nonnegative integer physical model.  In some cases, there is a natural way to assign measuring operations to these symbols.  For example, let $\mathcal{A}$ be a nonnegative integer physical model with a structure $\mathfrak{A}$, and with a symbol $f$ for an observable quantity.   An expansion $\mathcal{A}^\prime$, with a structure $\mathfrak{A}^\prime$, can be constructed by introducing a new symbol $g$ for an observable quantity.  Define
\begin{equation*}
g^{\mathfrak{A}^\prime}=h\circ f^\mathfrak{A},
\end{equation*}
where $h$ is a recursive partial function whose domain includes $\mathrm{ran}(\,f^\mathfrak{A})$.  We say that this observable quantity is \emph{derived} from $f^\mathfrak{A}$.  Note, in this case, that $\mathfrak{A}^\prime$ is a definitional expansion of $\mathfrak{A}$.

But because $h$ is a recursive partial function, there is an effective procedure for calculating $h$.  We can then define $\mathrm{op}(g)$ to be the following two-step measuring operation:
\begin{description}
\item[Step 1] Perform the operation $\mathrm{op}(f)$ to obtain a measurement result $m$.
\item[Step 2] Apply the procedure for $h$ to calculate $h(m)$.  This is the measurement result for $\mathrm{op}(g)$.
\end{description}
We say that this is a \emph{natural measuring operation} for the derived observable quantity.  But note that if model $\mathcal{A}$ is not faithful, then a measurement result of $\mathrm{op}(f)$ might be outside the range of $f^\mathfrak{A}$, and the effective procedure in Step~2 might not produce any result in a finite number of steps.  If Step~2 does not produce a result in a finite number of steps, then $\mathrm{op}(g)$ fails to produce a measurement result.\footnote{
Of course, an individual performing this operation might never know that it fails.  See Davis~\cite[p.\ 10]{Davis1958}, for example.
}

\begin{theorem}\label{T:derived}
Let $\mathcal{A}$ be any nonnegative integer physical model, and let $\mathcal{A}^\prime$ be an expansion of $\mathcal{A}$ that is obtained by introducing a derived observable quantity with a natural measuring operation.  Then, $\mathcal{A}^\prime$ is faithful if and only if $\mathcal{A}$ is faithful.
\end{theorem}
\begin{proof}
Suppose that $\mathcal{A}^\prime$ is faithful.  Because $\mathcal{A}$ is a reduct of $\mathcal{A}^\prime$, it follows from Corollary~\ref{C:reduct} that $\mathcal{A}$ is faithful.  Conversely, suppose that $\mathcal{A}$ is faithful.  Let $\mathfrak{A}$ and $\mathfrak{A}^\prime$ be the structures of $\mathcal{A}$ and $\mathcal{A}^\prime$, respectively.  Consider any symbol $g$ for an observable quantity of $\mathcal{A}^\prime$, consider any nonnegative integer $n$, and suppose that a measurement result of $\mathrm{op}(g)$ is equal to $n$.  There are two cases to consider.
\begin{description}
\item[Case 1] If $g$ is a symbol for an observable quantity of $\mathcal{A}$, then because $\mathcal{A}$ is faithful, there exists an $s\in\lvert\mathfrak{A}\rvert_S=\lvert\mathfrak{A}^\prime\rvert_S$ such that $n=g^\mathfrak{A}(s)=g^{\mathfrak{A}^\prime}(s)$.
\item[Case 2] If $g$ is the symbol for the derived observable quantity that was introduced to $\mathcal{A}^\prime$, then  $g^{\mathfrak{A}^\prime}=h\circ f^\mathfrak{A}$, where $f$ is a symbol for an observable quantity of $\mathcal{A}$, and where $h$ is a recursive partial function whose domain includes $\mathrm{ran}(\,f^\mathfrak{A})$.  By the definition of $\mathrm{op}(g)$, there is a measurement result $m$ of $\mathrm{op}(f)$ such that $h(m)=n$.  And because $\mathcal{A}$ is faithful, there exists an $s\in\lvert\mathfrak{A}\rvert_S=\lvert\mathfrak{A}^\prime\rvert_S$ such that $m=f^\mathfrak{A}(s)$.  Thus, there is an $s\in\lvert\mathfrak{A}^\prime\rvert_S$ such that
\begin{equation*}
n=h(m)=h\circ f^\mathfrak{A}(s)=g^{\mathfrak{A}^\prime}(s).
\end{equation*}
\end{description}
In either case, there exists an $s\in\lvert\mathfrak{A}^\prime\rvert_S$ such that $n=g^{\mathfrak{A}^\prime}(s)$.  By Definition~\ref{D:faithful}, $\mathcal{A}^\prime$ is faithful.
\end{proof}

As another application of definitional expansions, consider a structure $\mathfrak{A}$ for a nonnegative integer physical model $\mathcal{A}$.  Let $\mathfrak{A}^\prime$ be the definitional expansion of $\mathfrak{A}$ that is obtained by introducing, for each symbol $f$ for an observable quantity, a new predicate symbol $P_f$ of sort $N$ that has the following definition:
\begin{equation}\label{E:observational}
\forall_N\,x\,\Bigl(\,P_f(x)\;\leftrightarrow\;\exists_S\,s\,\bigl(\,f(s)=x\,\bigr)\,\Bigr).
\end{equation}
We say that these are the \emph{observational predicate symbols} introduced into the language.  Now let the \emph{observational structure} of $\mathcal{A}$ be the reduct of $\mathfrak{A}^\prime$ whose only sort is the sort $N$, and whose only symbols are the observational predicate symbols, together with the symbols of first-order number theory.  We have the following immediate corollary.
\begin{corollary}
Let $\mathcal{A}$ and $\mathcal{B}$ be any nonnegative integer physical models that have the same language and measuring operations.  Then, $\mathcal{A}$ and $\mathcal{B}$ are observationally equivalent if and only if they have the same observational structures.
\end{corollary}
\begin{proof}
Let $\mathfrak{A}$ and $\mathfrak{B}$ be the structures of $\mathcal{A}$ and $\mathcal{B}$, respectively.  By formula~\eqref{E:observational}, $P_f^{\mathfrak{A}^\prime}=\mathrm{ran}(\,f^\mathfrak{A})$ for each symbol $f$ for an observable quantity, and similarly for $\mathfrak{B}$.  Therefore, $\mathcal{A}$ is observationally equivalent to $\mathcal{B}$ if and only if, for each symbol $f$ for an observable quantity, $P_f^{\mathfrak{A}^\prime}=P_f^{\mathfrak{B}^\prime}$.  That is, $\mathcal{A}$ is observationally equivalent to $\mathcal{B}$ if and only if $\mathcal{A}$ and $\mathcal{B}$ have the same observational structures.
\end{proof}

\section{Restrictions of Models}\label{S:restrictions}

Recall that Model~\ref{M:baryon} is a nonnegative integer physical model with a structure $\mathfrak{A}$ such that $\lvert\mathfrak{A}\rvert_S=\mathbb{N}$, with a symbol $f$ for an observable quantity such that $f^\mathfrak{A}(s)=2s+2$, and with a measuring operation $\mathrm{op}(f)$ that counts the total number of baryons and antibaryons produced in a collision of two protons.  Assuming the law of baryon number conservation, Model~\ref{M:baryon} is faithful.

Now consider a submodel of Model~\ref{M:baryon}.  In particular, let $\mathcal{B}$ be the submodel of Model~\ref{M:baryon} with $\lvert\mathfrak{B}\rvert_S=\mathbb{N}-\{0\}$, where $\mathfrak{B}$ denotes the structure of $\mathcal{B}$.  Although $\mathrm{ran}(\,f^\mathfrak{A})$ contains all even positive integers, note that $\mathrm{ran}(\,f^\mathfrak{B})$ only contains those even positive integers that are \emph{greater than two}.  By the definition of a submodel (Definition~\ref{D:isomorphic}), both models have the same measuring operation.  But because collisions of protons have been observed\footnote{
For example, see Batson and Riddiford~\cite{Batson1956}.
} where the total number of baryons and antibaryons produced is equal to $2$, model $\mathcal{B}$ is not faithful.  Hence, a faithful model such as Model~\ref{M:baryon} may have a submodel that is not faithful.

Now consider a nonnegative integer physical model $\mathcal{C}$ that is identical to $\mathcal{B}$, except that $\mathrm{op}(f)$ is only intended to be performed when the total number of baryons and antibaryons produced is greater than two, and $\mathrm{op}(f)$ fails to produce a measurement result if this is not the case.  In contrast to model $\mathcal{B}$, model $\mathcal{C}$ is faithful.  We say that model $\mathcal{C}$ is a \emph{restriction} of Model~\ref{M:baryon} to the positive integers greater than two.  The concept of a restriction of a nonnegative integer physical model can be formalized in the following manner.  In this definition, we use $\mathrm{op}_\mathcal{A}(f)$ to denote the measuring operation assigned to $f$ in a model $\mathcal{A}$, and $\mathrm{op}_\mathcal{B}(f)$ to denote the measuring operation assigned to $f$ in a model $\mathcal{B}$.  Note that for each recursively enumerable set $Q$, there is~\cite[p.\ 63]{Enderton2001} a semidecision procedure for testing membership in $Q$.
\begin{definition}\label{D:restriction}
Let $\mathcal{A}$ and $\mathcal{B}$ be nonnegative integer physical models that have the same language, and that have $f$ as the only symbol for an observable quantity.  Let $\mathfrak{A}$ and $\mathfrak{B}$ be the structures of $\mathcal{A}$ and $\mathcal{B}$, respectively.  And let $Q$ be a recursively enumerable set of nonnegative integers such that $\mathrm{ran}(\,f^\mathfrak{A})\cap Q\ne\varnothing$.  Then, we say that $\mathcal{B}$ is a \emph{restriction} of $\mathcal{A}$ to the set $Q$ if and only if
\begin{enumerate}
\renewcommand{\theenumi}{\roman{enumi}}
\item\label{P:rest-structure} $\mathfrak{B}$ is the substructure of $\mathfrak{A}$  such that $\lvert\mathfrak{B}\rvert_S=\bigl\{s\in\lvert\mathfrak{A}\rvert_S\bigm\vert f^\mathfrak{A}(s)\in Q\bigr\}$; and
\item\label{P:rest-operation} $\mathrm{op}_\mathcal{B}(f)$ is the following two-step measuring operation:
\begin{description}
\item[Step 1] Perform the operation $\mathrm{op}_\mathcal{A}(f)$ to obtain a measurement result $n$.
\item[Step 2] Use the semidecision procedure for $Q$ to test whether $n\in Q$.  If this procedure verifies that $n\in Q$, then let $n$ be the measurement result for $\mathrm{op}_\mathcal{B}(f)$.  Otherwise, $\mathrm{op}_\mathcal{B}(f)$ fails to produce a measurement result.
\end{description}
\end{enumerate}
\end{definition}

One corollary of this definition is that if $\mathcal{A}$ is a computable physical model, then any restriction of $\mathcal{A}$ to a recursively enumerable set is also a computable physical model.  We also have the following corollary.
\begin{corollary}\label{C:restriction}
Let $\mathcal{A}$ be a nonnegative integer physical model with a structure $\mathfrak{A}$, and with $f$ as the only symbol for an observable quantity.  Let $Q$ be a recursively enumerable set of nonnegative integers such that $\mathrm{ran}(\,f^\mathfrak{A})\cap Q\ne\varnothing$, and let $\mathcal{B}$ be a restriction of $\mathcal{A}$ to the set $Q$.  If $\mathcal{A}$ is faithful, then $\mathcal{B}$ is faithful.
\end{corollary}
%
%
%
\begin{proof}
Let $\mathfrak{B}$ be the structure of $\mathcal{B}$.  Now suppose that $\mathcal{A}$ is faithful, consider any nonnegative integer $n$, and suppose that a measurement result of $\mathrm{op}_\mathcal{B}(f)$ is equal to $n$.  By part~(\ref{P:rest-operation}) of Definition~\ref{D:restriction}, it must be the case that $n\in Q$, and that $n$ is a measurement result for $\mathrm{op}_\mathcal{A}(f)$.  Thus, because $\mathcal{A}$ is faithful, there exists an $s\in\lvert\mathfrak{A}\rvert_S$ such that $f^\mathfrak{A}(s)=n$.  But then, by part~(\ref{P:rest-structure}) of Definition~\ref{D:restriction}, $s\in\lvert\mathfrak{B}\rvert_S$ and $f^\mathfrak{B}(s)=f^\mathfrak{A}(s)=n$.  Therefore, by Definition~\ref{D:faithful}, $\mathcal{B}$ is faithful.
\end{proof}

\section{An Example}\label{S:example}

At this point, an example that illustrates the use of some of the corollaries and theorems might be instructive.  Recall that Model~\ref{M:cannon} describes the trajectory of a projectile fired from a cannon at $5$ meters per second in an inertial reference frame, and in the absence of any external forces.  The model has a structure $\mathfrak{A}$ such that $\lvert\mathfrak{A}\rvert_S$ is the set of all $t\in\mathbb{N}$, the model has a symbol $f$ for an observable quantity such that $f^\mathfrak{A}(t)=J(t,5t)$, and the model has an operation $\mathrm{op}(f)$ that jointly measures the number of seconds $s$ since the projectile was fired, together with the number of meters $m$ between the cannon and the projectile at that time.

Let $K\colon\mathbb{N}\to\mathbb{N}$ and $L\colon\mathbb{N}\to\mathbb{N}$ be the recursive functions~\cite[p.\ 278]{Enderton2001}
%
%
that satisfy the equations
\begin{equation*}
K\bigl(J(a,b)\bigr)=a\quad\text{and}\quad L\bigl(J(a,b)\bigr)=b
\end{equation*}
for all nonnegative integers $a$ and $b$.  Now consider the following sequence of constructions.

For each nonnegative integer $u$, define $\mathcal{B}_u$ to be the restriction of Model~\ref{M:cannon} to the set $\{\,J(u,b)\mid b\in\mathbb{N}\,\}$.  Then $\lvert\mathfrak{B}_u\rvert_S=\{u\}$ and $f^{\mathfrak{B}_u}(u)=J(u,5u)$, where  $\mathfrak{B}_u$ denotes the structure of $\mathcal{B}_u$.  The measuring operation $\mathrm{op}_{\mathcal{B}_u}\!(f)$ produces results of the form $J(u,m)$, where $m$ is the number of meters between the cannon and the projectile, measured $u$ many seconds after the projectile was fired.  By Corollary~\ref{C:restriction}, $\mathcal{B}_u$ is faithful if Model~\ref{M:cannon} is faithful.\footnote{
But the converse of this statement does not necessarily hold.  For example, it might be the case that Model~\ref{M:cannon} is not faithful, but that $\mathcal{B}_u$ is faithful because there are insufficient resources in the universe to perform Step~2 in Definition~\ref{D:restriction}.
}

For each nonnegative integer $u$, define $\mathcal{C}_u$ to be the expansion of $\mathcal{B}_u$ that is obtained by introducing the derived observable quantity
\begin{equation*}
g_u^{\mathfrak{C}_u}=L\circ f^{\mathfrak{B}_u}
\end{equation*}
with a natural measuring operation, where $\mathfrak{C}_u$ denotes the structure of $\mathcal{C}_u$.  Note that $\lvert\mathfrak{C}_u\rvert_S=\{u\}$, $f^{\mathfrak{C}_u}(u)=J(u,5u)$, and $g_u^{\mathfrak{C}_u}(u)=5u$.  The measuring operation $\mathrm{op}(g_u)$ produces the number of meters between the cannon and the projectile, measured $u$ many seconds after the projectile was fired.  By Theorem~\ref{T:derived}, $\mathcal{C}_u$ is faithful if and only if $\mathcal{B}_u$ is faithful.

For each nonnegative integer $u$, let $\mathcal{D}_u$ be the reduct of $\mathcal{C}_u$ that has $g_u$ as its only symbol for an observable quantity.  By Corollary~\ref{C:reduct}, $\mathcal{D}_u$ is faithful if $\mathcal{C}_u$ is faithful.
%
%

For each nonnegative integer $u$, let $\mathcal{E}_u$ be a nonnegative integer physical model that has the same language and measuring operation as $\mathcal{D}_u$, and define its structure $\mathfrak{E}_u$ so that $\lvert\mathfrak{E}_u\rvert_S=\{0\}$ and $g_u^{\mathfrak{E}_u}(0)=5u$.  Then $\mathcal{E}_u$ is isomorphic to $\mathcal{D}_u$.  Therefore, $\mathcal{E}_u$ is faithful if and only if $\mathcal{D}_u$ is faithful.

And finally, let $\mathcal{F}$ be a nonnegative integer physical model that has $\{\,g_u\mid u\in\mathbb{N}\,\}$ as the set of symbols for its observable quantities, and that has no additional physical symbols.  Define $\mathcal{F}$ so that, for each nonnegative integer $u$, $\mathcal{E}_u$ is a reduct of $\mathcal{F}$.  Then $\lvert\mathfrak{F}\rvert_S=\{0\}$, where  $\mathfrak{F}$ denotes the structure of $\mathcal{F}$.  And for each nonnegative integer $u$, $g_u^\mathfrak{F}(0)=5u$ and $\mathrm{op}(g_u)$ produces the number of meters between the cannon and the projectile, measured $u$ many seconds after the projectile was fired.  By Corollary~\ref{C:one-symbol}, $\mathcal{F}$ is faithful if and only if, for each $u\in\mathbb{N}$, $\mathcal{E}_u$ is faithful.

In the context of Einstein's theory of relativity, the trajectory of a projectile is often thought of as a static line that exists in a single state in space-time.  The projectile's position at a particular time $u$ is then thought of as a property of the line that is measured by sampling a single point along the length of the line.  De Broglie~\cite[p.\ 114]{DeBroglie1949} described this conception as follows:
\begin{quote}
In space-time, everything which for each of us constitutes the past, the present, and the future is given in block, and the entire collection of events, successive for us, which form the existence of a material particle is represented by a line, the world-line of the particle. \ldots\  Each observer, as his time passes, discovers, so to speak, new slices of space-time which appear to him as successive aspects of the material world, though in reality the ensemble of events constituting space-time exist prior to his knowledge of them.
\end{quote}
\noindent
This static conception of the trajectory of the projectile in Model~\ref{M:cannon} is formalized by model $\mathcal{F}$ in the sense that model $\mathcal{F}$ has a single state, and there is a separate observable quantity for the position of the projectile at each time $u$.  Moreover, through the chain of implications in the preceding paragraphs, we have shown that if Model~\ref{M:cannon} is faithful, then $\mathcal{F}$ is faithful.

\section{Real Numbers}\label{S:real}

A commonly encountered form of measurement~\cite[Sect.\ 4.3.7]{JCGM2008} measures lower and upper bounds for a value,\footnote{
For any real numbers $b$ and $x$, we say that $b$ is a \emph{lower bound} for $x$ if and only if $b\leq x$, and we say that $b$ is a \emph{strict lower bound} for $x$ if and only if $b<x$.  Of course, every strict lower bound is also a lower bound, and similarly for upper bounds.
} rather than measuring the value directly.  For example, Perrin~\cite[Sect.\ 11]{Perrin1910} determined that the number of molecules in $2$ grams of molecular hydrogen gas\footnote{
Perrin~\cite[Sect.\ 6]{Perrin1910} defined Avogadro's constant to be  equal to this number, but a different definition~\cite[p.\ 134]{CCU2019} for Avogadro's constant is used nowadays.
} is strictly greater than $45\times10^{22}$, and strictly less than $200\times10^{22}$.  Hence, in units of septillions of molecules (that is, $10^{24}$ molecules), he determined that there are between $0.45$ septillion and $2.0$ septillion molecules in $2$ grams of molecular hydrogen gas.  Conventionally~\cite[pp.\ 13--16]{Taylor1997}, we express this by stating that Perrin measured the value to be equal to $1.2\pm0.8$ septillion molecules.  Conventions also require the lower and upper bounds to have only finitely many digits in their decimal expansions.  As a consequence, the bounds are rational numbers.

Now, an integer $i$ can be encoded as a nonnegative integer $\mathrm{int}(i)$ using the function
\begin{equation*}
\mathrm{int}(i)=\begin{cases}
2i &\text{if $i\geq0$}\\
-2i-1 &\text{otherwise}
\end{cases}.
\end{equation*}
And any rational number $\tfrac{a}{b}$ that is in lowest terms with $b>0$ can be encoded as a nonnegative integer $\mathrm{rat}\bigl(\tfrac{a}{b}\bigr)$ using
\begin{align*}
\mathrm{rat}\Bigl(\frac{a}{b}\Bigr)&=\\
&\mathrm{int}\Bigl((\mathrm{sgn}\:a)2^{\mathrm{int}(\alpha_1-\beta_1)} 3^{\mathrm{int}(\alpha_2-\beta_2)} 5^{\mathrm{int}(\alpha_3-\beta_3)} 7^{\mathrm{int}(\alpha_4-\beta_4)} 11^{\mathrm{int}(\alpha_5-\beta_5)}\cdots\Bigr),
\end{align*}
where
\begin{equation*}
a=(\mathrm{sgn}\:a)2^{\alpha_1}3^{\alpha_2}5^{\alpha_3}7^{\alpha_4}11^{\alpha_5}\cdots
\end{equation*}
and
\begin{equation*}
b=2^{\beta_1}3^{\beta_2}5^{\beta_3}7^{\beta_4}11^{\beta_5}\cdots
\end{equation*}
are the prime factorizations of $a$ and $b$, respectively.  We write $(p\,;q)$,
%
%
where $p<q$, to denote the open interval with endpoints $p$ and $q$.  We say that the interval is \emph{rational} if and only if $p$ and $q$ are both rational numbers.  A rational open interval $(p\,;q)$ can be encoded as a nonnegative integer $\mathrm{ival}(p\,;q)$ using
\begin{equation*}
\mathrm{ival}(p\,;q)=J\bigl(\mathrm{rat}(p),\mathrm{rat}(q)\bigr).
\end{equation*}
%
%
We use $\mathrm{ival}(p\,;q)$ to encode any joint measurement of a strict lower bound $p$ and a strict upper bound $q$.

For example, consider the claim that every electron has a mass of $m$ kilograms, where $m$ is some real number.  One plausible way to formalize this claim is to state that every strict lower bound that is measured for the mass of an electron is less than $m$ kilograms, and every strict upper bound measured for the mass is greater than $m$ kilograms.  A closely related way to formalize the claim
%
%
is to state that for every joint measurement of a strict lower bound $p$ and a strict upper bound $q$ for the electron's mass, $m\in(p\,;q)$.  Similarly, the claim can be formalized by asserting that the following nonnegative integer physical model is faithful.
\begin{model}
Consider a two-sorted nonnegative integer physical language with a symbol $f$ for an observable quantity, and with no additional physical symbols.  Let $\mathfrak{A}$ be a structure for this language, where $\lvert\mathfrak{A}\rvert_S$ is the set of all rational open intervals $(p\,;q)$ that contain the real number $m$.  Define $\mathrm{op}(f)$ to be an operation that jointly measures strict lower and upper bounds for the mass of an electron (in kilograms), and let $f^\mathfrak{A}(p\,;q)=\mathrm{ival}(p\,;q)$.
\end{model}
\noindent
Note that the states in this model can be thought of as corresponding to different states of the measurer, with different lower and upper bounds being measured in different states.

Now, any Cartesian product $(p_1\,;q_1)\times(p_2\,;q_2)\times\cdots\times(p_d\,;q_d)$ of open intervals is said to be an \emph{open rectangle}.
%
%
We say that the rectangle is \emph{rational} if and only if each interval is rational.  And we use the function
\begin{equation*}
\mathrm{rect}_d(i_1\times i_2\times\cdots\times i_d)=J\bigl(\mathrm{ival}(i_1),\mathrm{ival}(i_2),\ldots,\mathrm{ival}(i_d)\bigr),
\end{equation*}
where $d$ is a positive integer, to encode each rational open rectangle $i_1\times i_2\times\cdots\times i_d$ as a nonnegative integer.\footnote{
We define $J(a)=a$ for each $a\in\mathbb{N}$.  Hence, for the $d=1$ case, $\mathrm{rect}_1(i_1)=\mathrm{ival}(i_1)$.
}  In addition, we refer to any joint measurement of strict lower and upper bounds as a measurement of \emph{strict bounds}.

Now consider one mole of a gas at thermodynamic equilibrium in a sealed container.  The ideal gas law~\cite[pp.\ 9--10 \& 70]{Bowley1999} states that the gas's pressure $P$, volume $V$, and temperature $T$, in standard SI units, are real numbers that satisfy the equation
\begin{equation}\label{E:gas-law}
PV=N_Ak_BT,
\end{equation}
where $N_A$ and $k_B$ are constants.\footnote{
These are Avogadro's number and Boltzmann's constant, respectively.  By definition~\cite[pp.\ 133--134]{CCU2019}, $N_A$ is the integer $602214076\times10^{15}$, and the numerical value of $k_B$ is exactly equal to the rational number $1380649\times10^{-29}$ when the constant is expressed in SI units.
}  Given any point $\langle P,V,T\rangle$ in $\mathbb{R}^3$, the claim that $P$, $V$, and $T$ are the gas's pressure, volume, and temperature, respectively, can be formalized by stating that
\begin{equation*}
\langle P,V,T\rangle\in(p_1\,;q_1)\times(p_2\,;q_2)\times(p_3\,;q_3)
\end{equation*}
for every measurement $\mathrm{ival}(p_1\,;q_1)$ of strict bounds for the pressure, for every measurement $\mathrm{ival}(p_2\,;q_2)$ of strict bounds for the volume, and for every measurement $\mathrm{ival}(p_3\,;q_3)$ of strict bounds for the temperature.  Using a similar idea, the claim that the gas satisfies the ideal gas law can be formalized by asserting that the following nonnegative integer physical model is faithful.
\begin{model}\label{M:ideal-gas}
Consider a two-sorted nonnegative integer physical language with a symbol $f$ for an observable quantity, and with no additional physical symbols.  Let $\mathfrak{A}$ be a structure for this language, where $\lvert\mathfrak{A}\rvert_S$ is the set of all $\bigl\langle\langle P,V,T\rangle,r\bigr\rangle$ such that $P$, $V$, and $T$ are real numbers that satisfy equation~\eqref{E:gas-law}, and such that $r$ is a rational open rectangle containing the point $\langle P,V,T\rangle$.  Define $\mathrm{op}(f)$ to be an operation that measures, in SI units, strict bounds $b_1$ for the gas's pressure, strict bounds $b_2$ for the gas's volume, and strict bounds $b_3$ for the gas's temperature.  The result of this joint measurement is encoded as $J(b_1,b_2,b_3)$.  Let $f^\mathfrak{A}\bigl\langle\langle P,V,T\rangle,r\bigr\rangle=\mathrm{rect}_3(r)$.
\end{model}

Throughout the sciences, it is often the case that a model is specified by describing a nonempty closed set in $\mathbb{R}^d$, for some positive integer $d$.  The ideal gas law for one mole of gas is an example of such a model, since the set of all triples $\langle P,V,T\rangle$ of real numbers that satisfy equation~\eqref{E:gas-law} is a nonempty closed set in $\mathbb{R}^3$.  For any such nonempty closed set, we can define a \emph{complete basic neighborhood model} for the set, as follows.
\begin{definition}\label{D:cbnm-euclidean}
Let $d$ be a positive integer, and let $A$ be any nonempty closed set in $\mathbb{R}^d$.  We say that a nonnegative integer physical model with a structure $\mathfrak{A}$ is a \emph{complete basic neighborhood model} for $A$ if and only if
\begin{enumerate}
\renewcommand{\theenumi}{\roman{enumi}}
\item the language for $\mathfrak{A}$ has only one symbol $f$ for an observable quantity;
\item $\lvert\mathfrak{A}\rvert_S$ is the set of all ordered pairs $\langle\mathbf{x},r\rangle$ such that $\mathbf{x}\in A$ and such that $r$ is a rational open rectangle that contains $\mathbf{x}$; and
\item $f^\mathfrak{A}\langle\mathbf{x},r\rangle=\mathrm{rect}_d(r)$ for each $\langle\mathbf{x},r\rangle\in\lvert\mathfrak{A}\rvert_S$.
\end{enumerate}
\end{definition}

Note that a complete basic neighborhood model can have any measuring operation.  But typically, when a model is informally
%
%
specified by a nonempty closed set in $\mathbb{R}^d$, the $i$th component of each point in $\mathbb{R}^d$ is identified with a physical quantity $\mathcal{P}_i$, such as a pressure, volume, or temperature.  In this case, one may consider a measuring operation that jointly measures strict bounds $b_1$, $b_2$, \ldots, $b_d$ for $\mathcal{P}_1$, $\mathcal{P}_2$, \ldots, $\mathcal{P}_d$, respectively, with the result of the joint measurement encoded as $J(b_1,b_2,\ldots,b_d)$.  If such an operation exists, then the informal model can be formalized as a complete basic neighborhood model with this measuring operation.  Model~\ref{M:ideal-gas} formalizes the ideal gas law for one mole of gas in exactly this sense.

Another example is the simple harmonic oscillator model~\cite[Sect.\ 3.2]{Thornton2004} for a mass that is constrained to move in one dimension, and that is subject to a linear restoring force when the mass is displaced from its equilibrium position.
%
%
This model specifies that when measured in SI units, the amplitude $a$ of the mass's oscillations about the equilibrium position, the angular frequency $\omega$ of those oscillations,
%
%
the time $t$, a time $t_0$ at which the mass achieves maximum displacement,
%
%
and the displacement $x$ of the mass at time $t$ are all real numbers such that
\begin{equation*}
a\cos(\omega t-\omega t_0)=x.
\end{equation*}
The set of all quintuples $\langle a,\omega,t,t_0,x\rangle$ of real numbers satisfying this equation is a nonempty closed set in $\mathbb{R}^5$, and there is a measuring operation that jointly measures strict bounds for the five physical quantities in the model (that is, for the amplitude of the oscillations about the equilibrium position, the angular frequency of those oscillations, and so on).  Hence, the simple harmonic oscillator model can be formalized as a complete basic neighborhood model with this measuring operation.

\section{Topological Spaces}\label{S:t-spaces}

The formalism discussed in the previous section can be extended from $\mathbb{R}^d$ to more general topological spaces.\footnote{
We assume that the reader is familiar with the terminological conventions in, for example, McCarty's topology textbook~\cite{McCarty1988}.
}  In particular, a model can be specified by describing a nonempty closed set $A$ in a topological space $\langle X,\tau\rangle$.  In this more general setting, Weihrauch and Zhong~\cite[pp.\ 329--330]{Weihrauch2002} suggest that the elements in a countable basis or subbasis (assuming that such a basis or subbasis exists for the space) play a role that is analogous to that of the rational open intervals in the usual topology on $\mathbb{R}$.  Using this analogy, the definition of a complete basic neighborhood model can be generalized as follows.
\begin{definition}\label{D:cbnm-topological}
Let $A$ be any nonempty closed\thinspace\footnote{
%
%
One could generalize the definition of a complete basic neighborhood model to allow sets $A$ that are not closed, but any such model formalizes the informal model specified by $A$ in a much weaker sense.  In particular, if one removes the requirement that $A$ must be closed, then Theorem~\ref{T:cbnm-formalization} and Corollary~\ref{C:cbnm-equivalence} fail to hold.
} set in a topological space $\langle X,\tau\rangle$.  We say that a nonnegative integer physical model with a structure $\mathfrak{A}$ is a \emph{complete basic neighborhood model} for $A$ in the space $\langle X,\tau\rangle$ if and only if
\begin{enumerate}
\renewcommand{\theenumi}{\roman{enumi}}
\item the language for $\mathfrak{A}$ has only one symbol $f$ for an observable quantity;
\item $\langle X,\tau\rangle$ has a countable basis $\beta$;
\item\label{P:cbnm-states} $\lvert\mathfrak{A}\rvert_S$ is the set of all ordered pairs $\langle x,r\rangle$ such that $x\in A$ and such that $r$ is an element of $\beta$ that contains $x$; and
%
%
\item\label{P:cbnm-observable} $f^\mathfrak{A}\langle x,r\rangle=\mathrm{enc}(r)$ for each $\langle x,r\rangle\in\lvert\mathfrak{A}\rvert_S$, where $\mathrm{enc}\colon\beta\to\mathbb{N}$ is a one-to-one function for encoding the elements of $\beta$ as nonnegative integers.
\end{enumerate}
\end{definition}
\noindent
Note that the set of all rational open rectangles in $\mathbb{R}^d$ is a countable basis for the usual Euclidean topology on $\mathbb{R}^d$,
%
%
and $\mathrm{rect}_d$ is a one-to-one encoding for the basis.  Therefore, Definition~\ref{D:cbnm-euclidean} is a special case of Definition~\ref{D:cbnm-topological}.

As an example of a model specified by a nonempty closed set in a topological space, consider the claim that a sample of some chemical compound contains exactly $n$ molecules of the compound, where $n$ is a nonnegative integer.  We will use $\langle\mathbb{N},\delta\rangle$ to denote the set of all nonnegative integers with the discrete topology.  Every set of points is closed in $\langle\mathbb{N},\delta\rangle$, and the collection $\bigl\{\,\{a\}\bigm\vert a\in\mathbb{N}\,\bigr\}$ of all singleton sets is a countable basis for the space.  Define $\mathrm{sing}\bigl(\{a\}\bigr)=a$ for each $a\in\mathbb{N}$.  Then, the claim that the sample contains exactly $n$ molecules can be formalized by asserting that the following complete basic neighborhood model for the set $\{n\}$ in $\langle\mathbb{N},\delta\rangle$ is faithful.
\begin{model}
Consider a two-sorted nonnegative integer physical language with a symbol $f$ for an observable quantity, and with no additional physical symbols.  Let $\mathfrak{A}$ be a structure for this language where $\lvert\mathfrak{A}\rvert_S=\bigl\{\,\bigl\langle n,\{n\}\bigr\rangle\,\bigr\}$.  Define $\mathrm{op}(f)$ to be an operation that measures the exact number of molecules in the sample, and let $f^\mathfrak{A}\bigl\langle n,\{n\}\bigr\rangle=\mathrm{sing}\bigl(\{n\}\bigr)$.
\end{model}

Alternatively, the claim can be formalized using an operation that jointly measures lower and upper bounds for the number of molecules, rather than measuring the number of molecules exactly.  Note that
\begin{equation*}
\bigl\{\,\{a,a+1,a+2,\ldots,a+k\}\,\bigm\vert\,a\in\mathbb{N}\;\;\mathrm{and}\;\;k\in\mathbb{N}\,\bigr\}
\end{equation*}
is a countable basis for $\langle\mathbb{N},\delta\rangle$, and define
\begin{equation*} \mathrm{seg}\bigl(\{a,a+1,a+2,\ldots,a+k\}\bigr)=J(a,k)
\end{equation*}
%
%
for all nonnegative integers $a$ and $k$.  The claim can then be formalized by asserting that the following complete basic neighborhood model for $\{n\}$ in $\langle\mathbb{N},\delta\rangle$ is faithful.
\begin{model}\label{M:molecules-bounds}
Consider a two-sorted nonnegative integer physical language with a symbol $f$ for an observable quantity, and with no additional physical symbols.  Let $\mathfrak{A}$ be a structure for this language, where $\lvert\mathfrak{A}\rvert_S$ is the set of all
\begin{equation*}
\bigl\langle n,\{a,a+1,a+2,\ldots,a+k\}\bigr\rangle
\end{equation*}
such that $a$ and $k$ are nonnegative integers with $n\in\{a,a+1,a+2,\ldots,a+k\}$.  Define $\mathrm{op}(f)$ to be an operation that measures a lower bound $b$ and an upper bound $b+m$ for the number of molecules in the sample, where $b$ and $m$ are nonnegative integers.  The result of this joint measuring operation is encoded as $J(b,m)$.  Let $f^\mathfrak{A}\langle n,r\rangle=\mathrm{seg}(r)$ for each $\langle n,r\rangle\in\lvert\mathfrak{A}\rvert_S$.
\end{model}
\noindent
Note that the preceding two models have distinct bases, encodings, and measuring operations, despite the fact that they are both complete basic neighborhood models for the same set $\{n\}$ in the same space $\langle\mathbb{N},\delta\rangle$.

Now, the following theorem shows that when a complete basic neighborhood model is constructed for a nonempty closed set $A$ in a topological space, $A$ is uniquely determined by the set $\mathrm{ran}(\,f^\mathfrak{A})$ of possible measurement results that are allowed by the model.\footnote{
In this case, the function that maps $\mathrm{ran}(\,f^\mathfrak{A})$ to the set $A$ that is determined by $\mathrm{ran}(\,f^\mathfrak{A})$ is closely related to the inner representation for closed sets, as described by Weihrauch and Grubba~\cite[pp.\ 1388--1389]{Weihrauch2009}.
}
\begin{theorem}\label{T:cbnm-formalization}
Let $\mathfrak{A}$ be the structure of a complete basic neighborhood model for a nonempty closed set $A$ in a topological space $\langle X,\tau\rangle$ with a basis $\beta$ and encoding $\mathrm{enc}$.  Let $f$ be the symbol for the observable quantity in the language of $\mathfrak{A}$.  Then, for each point $x$ in $X$, $x\in A$ if and only if
\begin{equation}\label{E:cbnm-formalization}
\bigl\{\,\mathrm{enc}(r)\,\bigm\vert\,r\in\beta\;\;\mathrm{and}\;\;x\in r\,\bigr\}\subseteq\mathrm{ran}(\,f^\mathfrak{A}).
\end{equation}
\end{theorem}
\begin{proof}
Consider any $x\in X$, and suppose that $x\in A$.  By part~(\ref{P:cbnm-states}) of Definition~\ref{D:cbnm-topological},
\begin{equation*}
\bigl\{\,\langle x,r\rangle\,\bigm\vert\,r\in\beta\;\;\mathrm{and}\;\;x\in r\,\bigr\}\subseteq\lvert\mathfrak{A}\rvert_S.
\end{equation*}
Therefore, by part~(\ref{P:cbnm-observable}) of Definition~\ref{D:cbnm-topological}, condition~\eqref{E:cbnm-formalization} holds.  We have shown that $x\in A$ implies condition~\eqref{E:cbnm-formalization}.

Conversely, suppose that condition~\eqref{E:cbnm-formalization} holds.  By part~(\ref{P:cbnm-observable}) of Definition~\ref{D:cbnm-topological}, it must be the case that for each basis element $r$ that contains $x$, there is a $y\in A$ such that $\langle y,r\rangle\in\lvert\mathfrak{A}\rvert_S$.  Hence, by part~(\ref{P:cbnm-states}) of Definition~\ref{D:cbnm-topological}, for each basis element $r$ that contains $x$, there is a $y\in A$ such that $y\in r$.  That is, every basis element that contains $x$ intersects $A$.  So every neighborhood of $x$ intersects $A$.  Hence, $x$ is in the closure of $A$.
%
%
But since $A$ is closed, $x\in A$.  We have shown that condition~\eqref{E:cbnm-formalization} implies $x\in A$.
\end{proof}

\noindent
This theorem has the following corollary.
\begin{corollary}\label{C:cbnm-equivalence}
Let $\mathcal{A}$ be a complete basic neighborhood model for a nonempty closed set $A$ in a topological space $\langle X,\tau\rangle$, let $\mathcal{B}$ be a complete basic neighborhood model for a nonempty closed set $B$ in $\langle X,\tau\rangle$, and let $\mathcal{A}$ and $\mathcal{B}$ have the same basis $\beta$, the same encoding $\mathrm{enc}$, the same language, and the same measuring operation.  Then, $A=B$ if and only if $\mathcal{A}$ is observationally equivalent to $\mathcal{B}$.
\end{corollary}
%
%

Ideas from Weihrauch's theory of type-two effectivity~\cite{Weihrauch2000} can be used to define computable functions from one topological space into another topological space.  In particular, given any topological space $\langle X,\tau\rangle$ with a countable basis $\beta$ and with a one-to-one function $\mathrm{enc}\colon\beta\to\mathbb{N}$ for encoding the elements of $\beta$, we say that a function $\phi\colon\mathbb{N}\to\mathbb{N}$ is an \emph{oracle} for a point $x\in X$ if and only if
\begin{equation*}
\mathrm{ran}(\phi)=\bigl\{\,\mathrm{enc}(r)\,\bigm\vert\,r\in L\,\bigr\}
\end{equation*}
for some $L\subseteq\beta$ that is a local basis for $x$.  Let $\nu_\mathrm{enc}\colon\mathrm{ran}(\mathrm{enc})\to\beta$ be the function that is defined so that $\nu_\mathrm{enc}\bigl(\mathrm{enc}(r)\bigr)=r$ for each $r\in\beta$.  Thus, if $c$ is any nonnegative integer in $\mathrm{ran}(\mathrm{enc})$, then $\nu_\mathrm{enc}(c)$ is the basis element that is encoded by $c$.  We say that an oracle $\phi$ is \emph{nested} if and only if $\nu_\mathrm{enc}\bigl(\phi(i+1)\bigr)\subseteq\nu_\mathrm{enc}\bigl(\phi(i)\bigr)$ for each $i\in\mathbb{N}$.  And we say that the encoding $\mathrm{enc}\colon\beta\to\mathbb{N}$ has a \emph{recursively enumerable subset relation} if and only if the set
\begin{equation*}
\bigl\{\,J(c_1,c_2)\,\bigm\vert\,\nu_\mathrm{enc}(c_1)\subseteq\nu_\mathrm{enc}(c_2)\;\;\mathrm{and}\;\;c_1\in\mathrm{ran}(\mathrm{enc})\;\;\mathrm{and}\;\;c_2\in\mathrm{ran}(\mathrm{enc})\,\bigr\}
\end{equation*}
%
%
is recursively enumerable.

Now let $\langle X_1,\tau_1\rangle$ and $\langle X_2,\tau_2\rangle$ be topological spaces with countable bases $\beta_1$ and $\beta_2$, respectively, and with one-to-one encodings $\mathrm{enc}_1\colon\beta_1\to\mathbb{N}$ and $\mathrm{enc}_2\colon\beta_2\to\mathbb{N}$.  Given these bases and encodings, we say that a function $g$ from $\langle X_1,\tau_1\rangle$ into $\langle X_2,\tau_2\rangle$ is \emph{computable} if and only if there is a recursive partial function $h$ of one function variable and one number variable\footnote{
See Rogers~\cite[p.\ 347]{Rogers1987} for the definition of a recursive partial function of one function variable and one number variable, and for related notation.
%
%
} such that if $\phi$ is a nested oracle for any point $x$ in $X_1$, then $\lambda m[\,h(\phi,m)\,]$ is a nested oracle for the point $g(x)$ in $X_2$.\footnote{
Note that if the space $\langle X_2,\tau_2\rangle$ is $T_0$, then $g(x)$ is uniquely determined by the oracle $\lambda m[\,h(\phi,m)\,]$.  But if $\langle X_2,\tau_2\rangle$ is not $T_0$, then this might not be the case.
}  This definition of a computable function from $\langle X_1,\tau_1\rangle$ into $\langle X_2,\tau_2\rangle$ generalizes the Grzegorczyk-Lacombe definition~\cite{Grzegorczyk1957} of a computable function from $\mathbb{R}$ into $\mathbb{R}$.

In particular, the set of all rational open intervals is a countable basis for the usual topology on $\mathbb{R}$, and $\mathrm{ival}$ is a one-to-one encoding of these basis elements.  In this context, a function $\phi\colon\mathbb{N}\to\mathrm{ran}(\mathrm{ival})$ is a nested oracle for a real number $x$ if and only if
\begin{equation*}
\nu_\mathrm{ival}\bigl(\phi(0)\bigr),\;\;\nu_\mathrm{ival}\bigl(\phi(1)\bigr),\;\;\nu_\mathrm{ival}\bigl(\phi(2)\bigr),\;\;\ldots
\end{equation*}
is a sequence of nested intervals that form a local basis for $x$.  The Grzegorczyk-Lacombe computable functions are then the functions from $\mathbb{R}$ into $\mathbb{R}$ that are computable using these oracles for points in $\mathbb{R}$.\footnote{
See Weihrauch~\cite[p.\ 251]{Weihrauch2000} and Szudzik~\cite[Thms.\ 11.2 \& 11.3]{Szudzik2013}.
}

In many cases, the nonempty closed sets that are used to specify models are the graphs of functions.  The following theorem shows that if the graph of a computable function is a nonempty closed set, and if the encodings associated with the function have recursively enumerable subset relations, then the graph has a complete basic neighborhood model that is observationally equivalent to a computable physical model.
\begin{theorem}\label{T:graph}
Let $\langle X_1,\tau_1\rangle$ and $\langle X_2,\tau_2\rangle$ be topological spaces with countable bases $\beta_1$ and $\beta_2$, respectively, and with one-to-one encodings $\mathrm{enc}_1\colon\beta_1\to\mathbb{N}$ and $\mathrm{enc}_2\colon\beta_2\to\mathbb{N}$ with recursively enumerable subset relations.  Let $g$ be a computable function from $\langle X_1,\tau_1\rangle$ into $\langle X_2,\tau_2\rangle$ with these bases and encodings, and let the graph of $g$ be a nonempty closed set in the product topology on $X_1\times X_2$.  Then, any complete basic neighborhood model for the graph that has the basis $\beta_{1,2}=\{\,r_1\times r_2\mid\linebreak[0]
%
%
r_1\in\beta_1\;\;\mathrm{and}\;\;r_2\in\beta_2\,\}$ and encoding $\mathrm{enc}_{1,2}(r_1\times r_2)=J\bigr(\mathrm{enc}_1(r_1),\mathrm{enc}_2(r_2)\bigr)$ is observationally equivalent to a computable physical model.\footnote{
\label{N:encoding}The theorem also holds more generally for encodings of the form $\mathrm{enc}_{1,2}(r_1\times r_2)=e\bigr(\mathrm{enc}_1(r_1),\mathrm{enc}_2(r_2)\bigr)$, where $e$ is any one-to-one recursive partial function whose domain includes $\mathrm{ran}(\mathrm{enc}_1)\times\mathrm{ran}(\mathrm{enc}_2)$.
}
\end{theorem}
\begin{proof}
Let $A=\bigl\{\,\bigl\langle x,g(x)\bigr\rangle\bigm\vert x\in X_1\,\bigr\}$ be the graph of $g$, and let $\mathfrak{A}$ be the structure of any complete basic neighborhood model for $A$ that has the basis $\beta_{1,2}$ and encoding $\mathrm{enc}_{1,2}$.  Let $f$ be the symbol for the observable quantity in the language of $\mathfrak{A}$.  Then let $B$ be the set of all nonnegative integers of the form
\begin{equation*}
J\bigl(a,b,J(c_0,c_1,\ldots,c_b)\bigr),
\end{equation*}
where $a$ and $b$ are nonnegative integers with $a\leq b$, and where $c_0$, $c_1$, \ldots, $c_b$ are members of $\mathrm{ran}(\mathrm{enc}_1)$ such that the basis element encoded by $c_{i+1}$ is a subset of the basis element encoded by $c_i$ (that is, $\nu_{\mathrm{enc}_1}(c_{i+1})\subseteq\nu_{\mathrm{enc}_1}(c_i)$) for each nonnegative integer $i<b$.  Because the encoding $\mathrm{enc}_1$ has a recursively enumerable subset relation, $B$ is a recursively enumerable set.  For each finite sequence $\langle c_0,c_1,\ldots,c_b\rangle$ of nonnegative integers, and for each nonnegative integer $i$, define
\begin{equation*}
\mathrm{seq}_{\langle c_0,c_1,\ldots,c_b\rangle}(i)=\begin{cases}
c_i &\text{if $i\leq b$}\\
0 &\text{otherwise}
\end{cases}.
\end{equation*}
Now, because $g$ is a computable function, there exists a recursive partial function $h$ of one function variable and one number variable such that if $\phi$ is a nested oracle for any point $x$ in $X_1$, then $\lambda m[\,h(\phi,m)\,]$ is a nested oracle for $g(x)$.  Choose an effective procedure for calculating the function $h$.  Then let $p$ be the recursive partial function such that, for each nonnegative integer $i$, $p(i)$ is calculated according to the following two-step procedure:
\begin{description}
\item[Step 1] Use a semidecision procedure for $B$ to verify that $i\in B$.  If $i\notin B$, then $p(i)$ is undefined.
\item[Step 2] Let $a$, $b$, $c_0$, $c_1$, \ldots, $c_b$ be nonnegative integers such that
\begin{equation*}
i=J\bigl(a,b,J(c_0,c_1,\ldots,c_b)\bigr).
\end{equation*}
Then use the procedure for calculating $h$ to verify that
\begin{enumerate}
\renewcommand{\theenumi}{\roman{enumi}}
\item $h(\mathrm{seq}_{\langle c_0,c_1,\ldots,c_b\rangle},a)$ is defined; and
\item $\mathrm{seq}_{\langle c_0,c_1,\ldots,c_b\rangle}$ is not given any input greater than $b$ in the course of the calculation for $h(\mathrm{seq}_{\langle c_0,c_1,\ldots,c_b\rangle},a)$.
\end{enumerate}
Let $p(i)=J\bigl(c_a,h(\mathrm{seq}_{\langle c_0,c_1,\ldots,c_b\rangle},a)\bigr)$ if both of these conditions hold.  Otherwise, $p(i)$ is undefined.
\end{description}
Because $p$ is a recursive partial function, the range of $p$ is a recursively enumerable set.
%
%
Let $C$ be this recursively enumerable set.

Now consider any $\bigl\langle x,g(x)\bigr\rangle\in A$ and let $\phi$ be a nested oracle for $x$.  Then $\lambda a[\,h(\phi,a)\,]$ is a nested oracle for $g(x)$.  Hence, $\lambda a\bigl[J\bigl(\phi(a),h(\phi,a)\bigr)\bigr]$ is a nested oracle for $\bigl\langle x,g(x)\bigr\rangle$.  For each $a\in\mathbb{N}$, let $b$ be an integer greater than or equal to $a$ such that $\phi$ is not given an input greater than $b$ in the course of the calculation for $h(\phi,a)$.  Because there can only be finitely many steps in the calculation for $h(\phi,a)$, such an integer $b$ must exist.  Then, for each nonnegative integer $i\leq b$, let $c_i=\phi(i)$.  It immediately follows that
\begin{equation*}
J\bigl(\phi(a),h(\phi,a)\bigr)=J\bigl(c_a,h(\mathrm{seq}_{\langle c_0,c_1,\ldots,c_b\rangle},a)\bigr).
\end{equation*}
Hence, $J\bigl(\phi(a),h(\phi,a)\bigr)\in C$ for each $a\in\mathbb{N}$.  And because $\lambda a\bigl[J\bigl(\phi(a),h(\phi,a)\bigr)\bigr]$ is an oracle for $\bigl\langle x,g(x)\bigr\rangle$, the set $\bigl\{\,\nu_{\mathrm{enc}_{1,2}}\bigl(J(\phi(a),h(\phi,a))\bigr)\bigm\vert a\in\mathbb{N}\,\bigr\}$ is a local basis for $\bigl\langle x,g(x)\bigr\rangle$.  We have shown that for each $\bigl\langle x,g(x)\bigr\rangle\in A$, there exists a $K\subseteq C$ such that $\{\,\nu_{\mathrm{enc}_{1,2}}(k)\mid k\in K\,\}$ is a local basis for $\bigl\langle x,g(x)\bigr\rangle$.

Next, consider any $J\bigl(c_a,h(\mathrm{seq}_{\langle c_0,c_1,\ldots,c_b\rangle},a)\bigr)$ in $C$.  Let $x$ be any point in $\nu_{\mathrm{enc}_1}(c_b)$, and let $\phi$ be a nested oracle for $x$ such that $\phi(i)=c_i$ for each nonnegative integer $i\leq b$.  Then $c_a=\phi(a)$, and because the functions $\mathrm{seq}_{\langle c_0,c_1,\ldots,c_b\rangle}$ and $\phi$ agree for all inputs that are less than or equal to $b$, $h(\mathrm{seq}_{\langle c_0,c_1,\ldots,c_b\rangle},a)=h(\phi,a)$.  Hence,
\begin{equation*}
J\bigl(c_a,h(\mathrm{seq}_{\langle c_0,c_1,\ldots,c_b\rangle},a)\bigr)=J\bigl(\phi(a),h(\phi,a)\bigr).
\end{equation*}
And because $\phi$ is an oracle for $x$, and $\lambda m[\,h(\phi,m)\,]$ is an oracle for $g(x)$, there exist basis elements $r_1\in\beta_1$ and $r_2\in\beta_2$ with $x\in r_1$ and $g(x)\in r_2$ such that $\phi(a)=\mathrm{enc}_1(r_1)$ and $h(\phi,a)=\mathrm{enc}_2(r_2)$.  Hence,
\begin{align*}
J\bigl(c_a,h(\mathrm{seq}_{\langle c_0,c_1,\ldots,c_b\rangle},a)\bigr)&=J\bigl(\phi(a),h(\phi,a)\bigr)\\
&=J\bigl(\mathrm{enc}_1(r_1),\mathrm{enc}_2(r_2)\bigr)\\
&=\mathrm{enc}_{1,2}(r_1\times r_2),
\end{align*}
where $r_1\times r_2\in\beta_{1,2}$ and $\bigl\langle x,g(x)\bigr\rangle\in r_1\times r_2$.  But $\bigl\langle x,g(x)\bigr\rangle\in A$.  Therefore, we have shown that for each $J\bigl(c_a,h(\mathrm{seq}_{\langle c_0,c_1,\ldots,c_b\rangle},a)\bigr)\in C$, there exists a point $\bigl\langle x,g(x)\bigr\rangle\in A$ such that
\begin{equation*}
\bigl\langle x,g(x)\bigr\rangle\in r_1\times r_2=\nu_{\mathrm{enc}_{1,2}}\bigl(J(c_a,h(\mathrm{seq}_{\langle c_0,c_1,\ldots,c_b\rangle},a))\bigr).
\end{equation*}

Because $\mathrm{enc}_1$ and $\mathrm{enc}_2$ have recursively enumerable subset relations, $\mathrm{enc}_{1,2}$ also has a recursively enumerable subset relation.
%
%
Let $D$ be the set of all $d\in\mathrm{ran}(\mathrm{enc}_{1,2})$ such that there exists a $c\in C$ with $\nu_{\mathrm{enc}_{1,2}}(c)\subseteq\nu_{\mathrm{enc}_{1,2}}(d)$.  Then $D$ is a recursively enumerable set because $C$ is a recursively enumerable set, and because $\mathrm{enc}_{1,2}$ has a recursively enumerable subset relation.  Now consider any $d\in D$.  By definition, there exists a $c\in C$ with $\nu_{\mathrm{enc}_{1,2}}(c)\subseteq\nu_{\mathrm{enc}_{1,2}}(d)$.  But we have shown that for every $c\in C$ there exists a point $\bigl\langle x,g(x)\bigr\rangle\in A$ such that $\bigl\langle x,g(x)\bigr\rangle\in\nu_{\mathrm{enc}_{1,2}}(c)$.  Hence, there exists a point $\bigl\langle x,g(x)\bigr\rangle\in A$ with
\begin{equation*}
\bigl\langle x,g(x)\bigr\rangle\in\nu_{\mathrm{enc}_{1,2}}(c)\subseteq\nu_{\mathrm{enc}_{1,2}}(d).
\end{equation*}
But by Definition~\ref{D:cbnm-topological}, $\mathrm{ran}(\,f^\mathfrak{A})$ is the set of all $d\in\mathrm{ran}(\mathrm{enc}_{1,2})$ such that there exists a point $\bigl\langle x,g(x)\bigr\rangle\in A$ with $\bigl\langle x,g(x)\bigr\rangle\in\nu_{\mathrm{enc}_{1,2}}(d)$.  Therefore, we have shown that $d\in D$ implies $d\in \mathrm{ran}(\,f^\mathfrak{A})$.

Conversely, suppose that $d\in \mathrm{ran}(\,f^\mathfrak{A})$.  By Definition~\ref{D:cbnm-topological}, there must exist a point $\bigl\langle x,g(x)\bigr\rangle\in A$ such that $\bigl\langle x,g(x)\bigr\rangle\in\nu_{\mathrm{enc}_{1,2}}(d)$.  But there is a set $K\subseteq C$ such that $\{\,\nu_{\mathrm{enc}_{1,2}}(k)\mid k\in K\,\}$ is a local basis for $\bigl\langle x,g(x)\bigr\rangle$.  Therefore, there must exist a $k\in C$ such that $\bigl\langle x,g(x)\bigr\rangle\in\nu_{\mathrm{enc}_{1,2}}(k)\subseteq\nu_{\mathrm{enc}_{1,2}}(d)$.  It immediately follows from the definition of $D$ that $d\in D$.  In addition to showing that $d\in D$ implies $d\in \mathrm{ran}(\,f^\mathfrak{A})$, we have now shown that $d\in \mathrm{ran}(\,f^\mathfrak{A})$ implies $d\in D$.  Hence, $D=\mathrm{ran}(\,f^\mathfrak{A})$.  We have established that $\mathrm{ran}(\,f^\mathfrak{A})$ is a recursively enumerable set.  Now let $\mathcal{B}$ be a nonnegative integer physical model that has the same language and measuring operation as the complete basic neighborhood model for $A$, and that has a structure $\mathfrak{B}$ which is defined so that $\lvert\mathfrak{B}\rvert_S=\mathrm{ran}(\,f^\mathfrak{A})$, and so that $f^\mathfrak{B}(s)=s$ for each $s\in\lvert\mathfrak{B}\rvert_S$.  It immediately follows that $\mathcal{B}$ is observationally equivalent to the complete basic neighborhood model for $A$.  And because $\lvert\mathfrak{B}\rvert_S$ is a recursively enumerable set, $\mathcal{B}$ is a computable physical model.
\end{proof}

Although the graph of a computable function is not necessarily a nonempty closed set, most commonly encountered functions have graphs that are nonempty and closed.  In particular, the graph of every continuous function from a topological space $\langle X_1,\tau_1\rangle$ into a Hausdorff space $\langle X_2,\tau_2\rangle$ is necessarily~\cite[p.\ 140]{Dugundji1966}
%
%
closed in the product topology on $X_1\times X_2$.  Since every computable function is continuous,
this implies that every computable function from a nonempty topological space into a Hausdorff space has a nonempty closed graph.

Unless explicitly stated otherwise, we will follow the convention in Section~\ref{S:real} and use the set of all rational open rectangles in $\mathbb{R}^d$ as a basis for the usual topology on $\mathbb{R}^d$, and we will use $\mathrm{rect}_d$ as an encoding for the basis.  Given this convention, consider any computable function $g\colon\mathbb{R}^c\to\mathbb{R}^d$, where $c$ and $d$ are positive integers.  Because $\mathbb{R}^c$ is nonempty, and because the usual Euclidean topology on $\mathbb{R}^d$ is Hausdorff,
%
%
the graph of $g$ is a nonempty closed set.  The encodings $\mathrm{rect}_c$ and $\mathrm{rect}_d$ also have recursively enumerable subset relations.  Therefore, given the graph of a computable function from $\mathbb{R}^c$ into $\mathbb{R}^d$, Theorem~\ref{T:graph} implies that any complete basic neighborhood model for the graph that uses our conventional basis and encoding for the usual topology on $\mathbb{R}^{c+d}$ is observationally equivalent to a computable physical model.
Various functions from $\mathbb{R}^c$ into $\mathbb{R}^d$ are known to be computable.\footnote{
See Pour-El and Richards~\cite[pp.\ 27 \& 30]{PourEl2016}, for example.
}  In particular, $g(P,V)=\tfrac{1}{N_Ak_B}PV$ is a computable function from $\mathbb{R}^2$ into $\mathbb{R}$, and the graph of this function is the set of all triples of real numbers satisfying equation~\eqref{E:gas-law}.  Hence, Theorem~\ref{T:graph} implies that the complete basic neighborhood model for one mole of an ideal gas (Model~\ref{M:ideal-gas}) is observationally equivalent to a computable physical model.  Similarly, $g(a,\omega,t,t_0)=a\cos(\omega t-\omega t_0)$ is a computable function from $\mathbb{R}^4$ into $\mathbb{R}$, so the complete basic neighborhood model for a simple harmonic oscillator, as described near the end of Section~\ref{S:real}, is observationally equivalent to a computable physical model.  Using this approach, many commonly encountered models can be formalized as computable physical models.

\section{Probabilities}\label{S:probabilities}

The language of \emph{elementary real analysis}\footnote{
See Rogers~\cite[p.\ 386]{Rogers1987} for an alternate formulation of elementary real analysis.
} is the two-sorted first-order language that can be obtained from first-order number theory by introducing a sort $R$ for real numbers, introducing a predicate symbol for equality of sort $\langle R,R\rangle$, introducing function symbols for addition and multiplication of sort $\langle R,R,R\rangle$, and introducing a function symbol $\mathrm{conv}_{N\to R}$ of sort $\langle N,R\rangle$.  The symbol $\mathrm{conv}_{N\to R}$ is intended to denote the function that maps each nonnegative integer $n$ to the corresponding real number that is numerically equivalent to $n$.

Many models in the sciences associate real number probabilities with physical phenomena.\footnote{
We assume that readers are familiar with basic terminology for probability and statistics as found, for example, in Taylor's textbook~\cite{Taylor1997}.
%
%
}  For example, atoms of the radioactive isotope copper-64 have been observed~\cite{Be2012} to undergo $\beta^-$ decay.  But other decay modes for copper-64, such as $\beta^+$ decay, have also been observed.  Given that an atom has decayed, its decay mode is usually modeled as having been chosen randomly, with the model assigning each decay mode a probability of having been chosen.  This probability is said to be the \emph{branching ratio} for the decay mode.  In a sample where several atoms have undergone radioactive decay, each atom's decay is modeled as an independent trial.  Hence, in a sample where $i$ many copper-64 atoms have decayed, the probability that $j$ of those decays were $\beta^-$ decays is given by a binomial distribution.  In particular, the probability is given by
\begin{equation*}
B_{i,b}(j)=\tbinom{i}{j}b^j(1-b)^{i-j},
\end{equation*}
where $b$ is the branching ratio for $\beta^-$ decay in copper-64.\footnote{
Note that we define $0^0=1$ for the binomial distribution.  See Knuth~\cite[p.\ 408]{Knuth1992}.
}  This model for $\beta^-$ decay in a sample of copper-64 can be formalized by the following nonnegative integer physical model, given any real number $b$ such that $0\leq b\leq1$.
\begin{model}\label{M:decay}
Consider a nonnegative integer physical language that has three sorts, $N$, $S$, and $R$, and that contains the symbols for elementary real analysis.  The only additional physical symbols in the language are a symbol $f$ for an observable quantity and a function symbol $p$ of sort $\langle S,R\rangle$.  Let $\mathfrak{A}_b$ be a structure for this language, where the nonlogical symbols of elementary real analysis are assigned their traditionally intended meanings, and where $\lvert\mathfrak{A}_b\rvert_S$ is the set of all $\langle i,j,q\rangle$ such that $i$ and $j$ are nonnegative integers with $i\geq j$, and $q=B_{i,b}(j)$.
%
%
Define $\mathrm{op}(f)$ to be an operation that measures the total number $m$ of copper-64 atoms in the sample that have undergone radioactive decay, together with the number $n$ of those atoms that have undergone $\beta^-$ decay.  The result of this joint measuring operation is encoded as $J(m,n)$.  Let $f^{\mathfrak{A}_b}\langle i,j,q\rangle=J(i,j)$, and let $p^{\mathfrak{A}_b}\langle i,j,q\rangle=q$.
\end{model}

Assuming that the total number of decays measured in the sample is always greater than or equal to the number of $\beta^-$ decays that are measured, Model~\ref{M:decay} is faithful.  This is because the model allows all possible measurement results of the form $J(i,j)$ where $i\geq j$.  Each state $\langle i,j,q\rangle\in\lvert\mathfrak{A}_b\rvert_S$ is also labeled with a real number $q$.  This number $q$ is a probability because, for each nonnegative integer $i$, the set $\{\,\langle j,q\rangle\mid\langle i,j,q\rangle\in\lvert\mathfrak{A}_b\rvert_S\,\}$ is the graph of a binomial distribution.  Hence, the statement that ``in a sample with a total of $i$ many decays, the model associates the probability $q$ with the possibility of $j$ many $\beta^-$ decays'' can be expressed\footnote{
Note that the function $J$ is definable in first-order number theory.
%
%
} in the language of the model as
\begin{equation*}
\exists_S\,s\,\bigl(\,f(s)=J(i,j)\;\wedge\;p(s)=q\,\bigr).
\end{equation*}

Statistical tests are often used to compare a model, such as Model~\ref{M:decay}, with its measurement results.\footnote{
Several commonly encountered statistical tests are described by Taylor~\cite[pp.\ 236--240 \& 271--277]{Taylor1997}.
%
%
The history of statistical tests in particle physics is surveyed by Franklin~\cite{Franklin2013}.
}  For example, consider a sample of copper-64 with a total of $m$ many decays.  Model~\ref{M:decay} associates a probability of $B_{m,b}(k)$ with the possibility that $k$ of those decays are $\beta^-$ decays, and this probability distribution has a mean value of $mb$.  Given a real number $\alpha$ such that $0<\alpha<1$,
%
%
%
and given a measurement result $J(m,n)$ for $\mathrm{op}(f)$, a \emph{two-tailed statistical test} of Model~\ref{M:decay} is conducted by comparing $\alpha$ with $P(m,n,b)$, where
\begin{equation}\label{E:p-sum}
P(m,n,b)=\hspace{-23pt}\sum_{\substack{k\in\{0,1,2,\ldots,m\}\\\lvert k-mb\rvert\,\geq\,\lvert n-mb\rvert}}\hspace{-24pt}B_{m,b}(k).
\end{equation}
In particular, the test is said to \emph{reject} Model~\ref{M:decay}
%
%
with a significance level of $\alpha$ if and only if\thinspace\footnote{
%
%
This convention for rejection is used, for example, by Taylor~\cite[pp.\ 237 \& 272]{Taylor1997}.
%
%
Some authors, such as Dekking et al.~\cite[Sect.\ 26.2]{Dekking2005}, use a different convention and would reject the model if and only if $P(m,n,b)\leq\alpha$.
}
\begin{equation*}
P(m,n,b)<\alpha.
\end{equation*}
Typically, the value chosen for $\alpha$ is close to zero.  In this case, a rejection of Model~\ref{M:decay} implies that in a sample with a total of $m$ many decays, the model only associates a small probability $P(m,n,b)$ with the possibility that the number of $\beta^-$ decays is at least as far from the mean as $n$.  That is, this possibility is improbable according to the model.

As an alternate way to formalize this statistical test, consider any real numbers $\alpha$ and $b$ such that $0<\alpha<1$ and $0\leq b\leq1$.  Let $\mathcal{B}_{\alpha,b}$ be a submodel of Model~\ref{M:decay} whose set of states is the set of all $\langle i,j,q\rangle\in\lvert\mathfrak{A}_b\rvert_S$ such that
\begin{equation*}
P(i,j,b)\geq\alpha,
\end{equation*}
and let $\mathfrak{B}_{\alpha,b}$ be the structure of $\mathcal{B}_{\alpha,b}$.
%
%
Then, for each measurement result $J(m,n)$ of $\mathrm{op}(f)$,\footnote{
Here we are assuming that $m\geq n$.
}
%
%
\begin{equation*}
P(m,n,b)<\alpha\quad\text{if and only if}\quad J(m,n)\notin\mathrm{ran}(\,f^{\mathfrak{B}_{\alpha,b}}).
\end{equation*}
That is, given any measurement result $J(m,n)$ for $\mathrm{op}(f)$, the two-tailed statistical test rejects Model~\ref{M:decay} with a significance level of $\alpha$ if and only if $J(m,n)\notin\mathrm{ran}(\,f^{\mathfrak{B}_{\alpha,b}})$.  This fact allows statements about the rejection of Model~\ref{M:decay} to be expressed as statements about $\mathcal{B}_{\alpha,b}$.  And in this sense, the model $\mathcal{B}_{\alpha,b}$ provides an alternate way to formalize the two-tailed statistical test.  For example, the statement that there exists a measurement result for which the statistical test rejects Model~\ref{M:decay} with a significance level of $\alpha$ can be expressed by stating that $\mathcal{B}_{\alpha,b}$ is not faithful.  Or equivalently, $P(m,n,b)\geq\alpha$ for every measurement result $J(m,n)$ of $\mathrm{op}(f)$ if and only if $\mathcal{B}_{\alpha,b}$ is faithful.

Given values for $\alpha$ and $b$, the faithfulness of $\mathcal{B}_{\alpha,b}$ is determined by the set $O_f$ that was introduced in Section~\ref{S:cp-models}.  For example, if $O_f$ is the set
\begin{equation}\label{E:universe}
\bigl\{\,J(m,n)\,\bigm\vert\,m\geq n\;\;\mathrm{and}\;\;m\in\mathbb{N}\;\;\mathrm{and}\;\;n\in\mathbb{N}\,\bigr\}
\end{equation}
then, for every choice of $\alpha$ and $b$,
%
%
there exists a measurement result $J(m,n)\in O_f$ such that $P(m,n,b)<\alpha$.
In this case, for every choice of $\alpha$ and $b$, the model $\mathcal{B}_{\alpha,b}$ is not faithful.  But, if $b=1$ and $O_f=\{\,J(m,m)\mid m\in\mathbb{N}\,\}$, then $P(m,n,b)=1$ for every $J(m,n)\in O_f$.  In this case, $\mathcal{B}_{\alpha,b}$ is faithful for every choice of $\alpha$.

As another example, if $0<b<1$ and if $O_f$ is any nonempty finite subset of~\eqref{E:universe}, then there exists a positive real number $\alpha$ such that
\begin{equation*}
\min\bigl\{\,P(m,n,b)\,\bigm\vert\,J(m,n)\in O_f\,\bigr\}>\alpha.
\end{equation*}
%
%
%
%
And if $O_f$ is the empty set, then it is vacuously true that $P(m,n,b)\geq\alpha$ for every $J(m,n)\in O_f$.  Therefore, if $0<b<1$ and if $O_f$ is any finite subset of~\eqref{E:universe}, then there exists an $\alpha$ such that $\mathcal{B}_{\alpha,b}$ is faithful.
%
%
Note that there are various plausible circumstances in which one might suppose that $O_f$ is finite.  In particular, if the entire observable universe is finite, then $O_f$ is necessarily a finite set.\footnote{
We say that the observable universe is finite if and only if every maximally faithful nonnegative integer physical model is observationally equivalent to a nonnegative integer physical model that has a finite set of states.
%
%
}

A \emph{statistical estimator} can be used to estimate a probability or other parameter that appears in a statistical model.\footnote{
Lyons~\cite[pp.\ 47--48]{Lyons2012} describes the methods of statistical parameter estimation that are most commonly used in particle physics.  See Casella and Berger~\cite[Chaps.\ 7 \& 9]{Casella2002} for the basic theory of statistical estimators.
%
%
}  For example, consider any measurement result $J(m,n)$ for $\mathrm{op}(f)$ where $m\geq n$, and consider any significance level $\alpha$.  Let $B_{m,n,\alpha}$ be the set of all $b$ in the closed interval $[\,0\,;1\,]$ such that the measurement result $J(m,n)$ is consistent with the claim that $\mathcal{B}_{\alpha,b}$ is faithful.  That is, let
\begin{equation*}
B_{m,n,\alpha}=\bigl\{\,b\in[\,0\,;1\,]\,\bigm\vert\,P(m,n,b)\geq\alpha\,\bigr\}.
\end{equation*}
From the assumption that $\mathcal{B}_{\alpha,b}$ is faithful for some $b\in[\,0\,;1\,]$, we may deduce that this $b$ is in the interval $[\,r\,;s\,]$, where $r$ is the greatest lower bound of $B_{m,n,\alpha}$, and $s$ is the least upper bound of $B_{m,n,\alpha}$.  Under that assumption, the interval $[\,r\,;s\,]$ provides an estimate for the branching ratio $b$.  For example, if $\alpha=1/3$ and $J(m,n)=J(3,2)$, then $B_{m,n,\alpha}$ is the set of all $b\in[\,0\,;1\,]$ such that $P(3,2,b)\geq\tfrac{1}{3}$.  It then follows (see Figure~\ref{F:p-3-2})
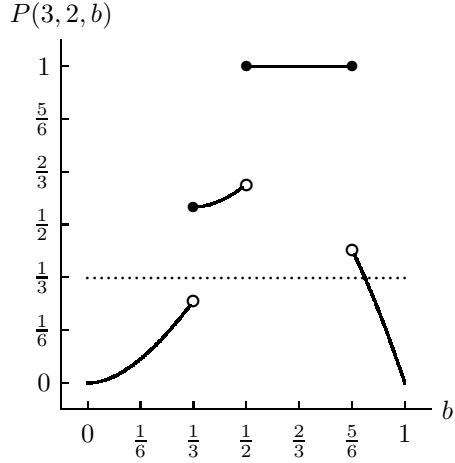
\begin{figure}[t]
\centering
\begin{picture}(166,177)(-28,-33)
%
%
\thinlines
\put(-10,-10){\line(1,0){140}}
\put(-10,-10){\line(0,1){140}}
%
\put(134,-10){\makebox(0,0)[l]{$b$}}
\put(-10,134){\makebox(0,0)[b]{$P(3,2,b)$}}
%
\thinlines
\multiput(0,-10)(20,0){7}{\line(0,1){3}}
\multiput(-10,0)(0,20){7}{\line(1,0){3}}
%
\put(0,-16){\makebox(0,0)[t]{$0$}}
\put(20,-16){\makebox(0,0)[t]{$\tfrac{1}{6}$}}
\put(40,-16){\makebox(0,0)[t]{$\tfrac{1}{3}$}}
\put(60,-16){\makebox(0,0)[t]{$\tfrac{1}{2}$}}
\put(80,-16){\makebox(0,0)[t]{$\tfrac{2}{3}$}}
\put(100,-16){\makebox(0,0)[t]{$\tfrac{5}{6}$}}
\put(120,-16){\makebox(0,0)[t]{$1$}}
%
\put(-14,0){\makebox(0,0)[r]{$0$}}
\put(-14,20){\makebox(0,0)[r]{$\tfrac{1}{6}$}}
\put(-14,40){\makebox(0,0)[r]{$\tfrac{1}{3}$}}
\put(-14,60){\makebox(0,0)[r]{$\tfrac{1}{2}$}}
\put(-14,80){\makebox(0,0)[r]{$\tfrac{2}{3}$}}
\put(-14,100){\makebox(0,0)[r]{$\tfrac{5}{6}$}}
\put(-14,120){\makebox(0,0)[r]{$1$}}
%
\multiput(0,40)(2.667,0){46}{\circle*{1}}
%
\thicklines
\put(40,31.111){\circle{4}}
\put(60,75.000){\circle{4}}
\put(100,50.556){\circle{4}}
%
\put(40,66.667){\circle*{4}}
\put(60,120.000){\circle*{4}}
\put(100,120.000){\circle*{4}}
%
\thicklines
\qbezier(0.000,0.000)(16.308,0.000)(38.794,29.516)
\qbezier(40.000,66.667)(48.278,66.667)(58.384,73.822)
\put(60,120){\line(1,0){40}}
\qbezier(100.859,48.750)(110.706,27.881)(120.000,0.000)
\end{picture}
\caption{The graph of $P(3,2,b)$ as a function of $b$.
%
%
Note that $P(3,2,b)\geq\tfrac{1}{3}$ if and only if the corresponding point in the graph is on or above the dotted line.  The graph intersects the line at $b=\sqrt[3]{\tfrac{2}{3}}$.  There are discontinuities at $b=\tfrac{1}{3}$, $\tfrac{1}{2}$, and $\tfrac{5}{6}$.}
\label{F:p-3-2}
\end{figure}
that
\begin{equation*} \bigl[\,r\,;s\,\bigr]=\Bigl[\,\tfrac{1}{3}\,;\sqrt[3]{\tfrac{2}{3}}\,\,\Bigr].
\end{equation*}
This function for mapping a measurement result $J(m,n)$ to an interval $[\,r\,;s\,]$ is an \emph{interval estimator} for the branching ratio in Model~\ref{M:decay}.\footnote{
Some other closely-related interval estimators are described, for example, by Crow~\cite{Crow1956}.
}  As we will show, there is an operation for measuring this interval estimate if $[\,r\,;s\,]$ is suitably encoded as a nonnegative integer.

First, consider any positive integer $m$, any nonnegative integers $n\leq m$ and $i<2m$, and any real numbers $b_1$ and $b_2$ in the open interval $\bigl(\tfrac{i}{2m}\,;\tfrac{i+1}{2m}\bigr)$.  Then, for every $k\in\{0,1,2,\ldots,m\}$,
\begin{equation*}
\lvert k-mb_1\rvert\geq\lvert n-mb_1\rvert\quad\text{if and only if}\quad\lvert k-mb_2\rvert\geq\lvert n-mb_2\rvert.
\end{equation*}
Hence, by equation~\eqref{E:p-sum}, there is a polynomial function $\psi_i$ such that $\psi_i(b)=P(m,n,b)$ for every $b\in\bigl(\tfrac{i}{2m}\,;\tfrac{i+1}{2m}\bigr)$.  That is, $P(m,n,b)$ is a piecewise polynomial function of $b$ where the polynomials' coefficients are rational numbers.  Also, for each nonnegative integer $i\leq2m$ there is a rational number $\rho_i$ such that $\rho_i=P(m,n,\tfrac{i}{2m})$.  Thus, the statement that $b\in B_{m,n,\alpha}$ can be expressed in the language of elementary real analysis as
\begin{equation}\label{E:F}
\bigvee_{i<2m}\Bigl(\,\psi_i(b)\geq\alpha\;\wedge\;\tfrac{i}{2m}<b\;\wedge\;b<\tfrac{i+1}{2m}\,\Bigr)\;\vee\;\bigvee_{i\leq2m}\Bigl(\,\rho_i\geq\alpha\;\wedge\;b=\tfrac{i}{2m}\,\Bigr),
\end{equation}
where $\alpha$ is a constant of sort $R$ and $b$ is a variable of sort $R$.\footnote{
Note that rational number constants and the $\geq$ and $<$ predicates are definable in elementary real analysis.
}  Moreover, there is an effective procedure that produces this formula given a positive integer $m$ and a nonnegative integer $n\leq m$.

Next, consider any $J(m,n)$ in the set $\mathrm{ran}(\,f^{\mathfrak{A}_b})$ of possible measurement results allowed by Model~\ref{M:decay}.  If $m=0$ then let $F_\alpha(b)$ denote the formula $0\leq b\;\wedge\;b\leq1$.  Otherwise, let $F_\alpha(b)$ denote formula~\eqref{E:F}.  The statement that $s$ is the least upper bound of $B_{m,n,\alpha}$ can be expressed in elementary real analysis as
\begin{equation}\label{E:lub}
\forall_R\,x\,\bigl(\,F_\alpha(x)\;\rightarrow\;x\leq s\,\bigr)\;\wedge\;\forall_R\,x\,\Bigl(\,x<s\;\rightarrow\;\exists_R\,y\,\bigl(\,F_\alpha(y)\;\wedge\;y>x\,\bigr)\,\Bigr).
\end{equation}
But if $\alpha$ is a rational number, then this is also a formula in the language of real closed fields.  In this case, the formula can be put into a prenex normal form and Collins' quantifier elimination algorithm~\cite{Collins1975} can be applied to obtain\footnote{
The set $B_{m,n,\alpha}$ is nonempty,
%
%
and every member of $B_{m,n,\alpha}$ has $1$ as an upper bound.  Hence, $B_{m,n,\alpha}$ has a unique least upper bound $s$.  In this case~\cite[pp.\ 151 \& 159--160]{Collins1975}, Collins' algorithm outputs the quantifier-free formula $\bigl(\,\varphi(s)=0\;\wedge\;c_2s-c_1>0\;\wedge\;d_2s-d_1<0\,\bigr)\;\vee\;1=0$.  This formula is equivalent to formula~\eqref{E:lub}.
} a squarefree polynomial $\varphi(x)$ that has integer coefficients, and to obtain a rational open interval $\bigl(\tfrac{c_1}{c_2}\,;\tfrac{d_1}{d_2}\bigr)$.  The polynomial has $s$ as a root, and the interval isolates this root.  Given the polynomial and isolating interval, a root refinement algorithm~\cite[Sect.\ 7]{Heindel1971} can be used to obtain a program for a nested oracle for $s$, written in a computationally universal programming language.  The G\"{o}del number $\overline{s}$ of this program
%
%
encodes $s$.  And a similar procedure can be used to obtain an encoding $\overline{r}$ of the greatest lower bound $r$.  Hence, for each rational number $\alpha$ such that $0<\alpha<1$, there is a recursive partial function $h$ that maps each $J(m,n)\in\mathrm{ran}(\,f^{\mathfrak{A}_b})$ to $J(\,\overline{r},\overline{s}\,)$.

Now, given any rational number $\alpha$ with $0<\alpha<1$, and given any real number $b$ with $0\leq b\leq1$, let $\mathcal{C}_{\alpha,b}$ to be the expansion of Model~\ref{M:decay} that is obtained by introducing the derived observable quantity
\begin{equation*}
g^{\mathfrak{C}_{\alpha,b}}=h\circ f^{\mathfrak{A}_b}
\end{equation*}
with a natural measuring operation, where $\mathfrak{C}_{\alpha,b}$ denotes the structure of $\mathcal{C}_{\alpha,b}$.  By the definition of a natural measuring operation, $\mathrm{op}(g)$ is an operation that measures the total number $m$ of copper-64 atoms that have undergone radioactive decay within a sample, together with the number $n$ of those atoms that have undergone $\beta^-$ decay, and then uses these values to calculate $J(\,\overline{r},\overline{s}\,)$.  This measurement result is an encoding of the closed interval $[\,r\,;s\,]$, and in this sense, $\mathrm{op}(g)$ measures the interval estimate of the branching ratio $b$.  Moreover, if $\mathcal{B}_{\alpha,b}$ is faithful then $b$ is contained within this interval.  Similar approaches can be used to measure interval estimates for parameters in other commonly encountered statistical models.

Incidentally, both Model~\ref{M:decay} and $\mathcal{C}_{\alpha,b}$ are isomorphic to computable physical models under the isomorphism that maps $\langle i,j,q\rangle$ to $J(i,j)$.
%
%

\section*{Acknowledgements}

We thank the anonymous reviewers for suggesting several improvements to this article.

\bibliographystyle{amsplain}
\bibliography{SemanticsOfCPM}

\end{document}